\documentclass[12pt,a4paper]{article}
\pdfoutput=1
\usepackage{hyperref}
\hypersetup{hidelinks,
          colorlinks=true,
          allcolors=black,
          pdfstartview=Fit,
          breaklinks=true} 
\usepackage{amsmath}
\numberwithin{equation}{section}
\usepackage{amssymb}
\usepackage{amsthm} 
\usepackage{bm}
\usepackage{color}
\usepackage{array}
\usepackage{geometry}
\usepackage{blindtext}
\usepackage{tabulary}
\usepackage{booktabs}
\usepackage{longtable}
\usepackage{threeparttable}
\usepackage{multirow}
\usepackage{makecell}
\usepackage{enumerate}
\usepackage{graphicx}
\usepackage{subfigure}
\usepackage{epstopdf}
\usepackage{float}
\usepackage[title]{appendix} 
\usepackage{embrac}
\usepackage{indentfirst}
\usepackage[numbers,sort&compress]{natbib} 
\newtheorem{thm}{\hspace{1em}Theorem}
\newtheorem{ass}{\hspace{1em}Assumption}
\newtheorem{lem}{\hspace{1em}Lemma}
\newtheorem{rem}{\hspace{1em}Remark}
\newtheorem{defi}{\hspace{1em}Definition}
\newtheorem{Prop}{\hspace{1em}Proposition}
\usepackage[linesnumbered,boxed,ruled,commentsnumbered]{algorithm2e}
\newcommand{\mR}{\mathbb{R}}
\newcommand{\G}{\mathcal{G}}
\newcommand{\V}{\mathcal{V}}
\newcommand{\E}{\mathcal{E}}
\newcommand{\mL}{\mathcal{L}}
\newcommand{\K}{\mathcal{K}}
\newcommand{\N}{\mathcal{N}}
\newcommand{\x}{\mathrm{x}}
\newcommand{\y}{\mathrm{y}}
\newcommand{\Y}{\mathrm{Y}}
\newcommand{\z}{\mathrm{z}}
\newcommand{\Z}{\mathrm{Z}}
\newcommand{\U}{\mathbf{U}}
\newcommand{\mH}{\mathbf{H}}
\newcommand{\D}{\mathbf{D}}

\newcommand{\1}{\mathbf{1}_m}

\newcommand{\dist}{\mathrm{dist}}
\newcommand{\mP}{\mathcal{P}}
\newcommand{\M}{\mathcal{M}}
\newcommand{\tabincell}[2]{\begin{tabular}{@{}#1@{}}#2\end{tabular}} 
\usepackage{caption, threeparttable}
\usepackage[font=small,labelfont=bf,labelsep=none]{caption}
\usepackage{authblk}
\title{\textbf{Decentralized Proximal Method of Multipliers for Convex Optimization with Coupled Constraints}}
\author[1]{Kai Gong }
\author[1]{Liwei Zhang}
\affil[1]{School of Mathematical Sciences, Dalian University of Technology, Dalian,
          116024, Liaoning, China. \authorcr Email: gk1995\_\_\_@mail.dlut.edu.cn
      \authorcr Email: lwzhang@dlut.edu.cn}
\date{}
\begin{document}
	\maketitle
	\begin{abstract}
		In this paper, a decentralized proximal method of multipliers (DPMM) is 
		proposed to solve constrained convex optimization problems 
		over multi-agent networks, where the local objective of each agent is a general closed convex function, and the constraints are coupled equalities and inequalities. This algorithm strategically integrates the dual decomposition method and the proximal point algorithm. One advantage of DPMM is that
	    subproblems can be solved inexactly and in parallel by agents at each iteration, 
	    which relaxes the restriction of requiring exact solutions to subproblems in many distributed constrained optimization algorithms. We show that the first-order optimality residual of the proposed algorithm decays to $0$ at a rate of $o(1/k)$ under general convexity. Furthermore, if a structural assumption for the considered optimization problem is satisfied, the sequence generated by DPMM converges linearly to an optimal solution. In numerical simulations, we compare DPMM with several existing algorithms using two examples to demonstrate its effectiveness.\\
		
		\noindent{\bf Keywords}:\ Distributed convex optimization, optimization 
		algorithm, sublinear convergence rate, linear convergence rate.
	\end{abstract}
	\section{Introduction}
	Distributed optimization algorithms decompose an optimization problem into smaller, more manageable subproblems that can be solved in parallel by a group of agents or processors. All agents communicate peer-to-peer and compute locally, cooperating to minimize the global objective function. Consequently, distributed optimization algorithms are widely used to solve large-scale problems in wireless communication, optimal control, machine learning, etc. Many problems of interest can be formulated as the following distributed constrained optimization model
	\begin{equation}\tag{\bf P}\label{P}
		\begin{aligned}
			&\min_\x F(\x)\triangleq\sum_{i=1}^mf_i(\x_i)\\
			{\rm s.t.}\quad& \sum_{i=1}^m\underbrace{\left[
				\begin{array}{c}
					A_i\x_i-b_i\\
					g_i(\x_i)
				\end{array}\right]}_{\triangleq \ G_i(\x_i)}
			\in\underbrace{\left[
				\begin{array}{c}
					0_p\\
					\mR^q_-
				\end{array}\right]}_{\triangleq \ \K},\\
			& \x_i\in\Omega_i, \ i=1,2,\dots,m,
		\end{aligned}
	\end{equation}
	where $\x_i\in\mR^{n_i}$ is the local decision variable controlled by agent $i$, 
	$\x=\mathrm{col} (\x_1,\x_2,\dots,\x_m)\in\mR^n$, $n=\sum_{i=1}^mn_i$. The objective 
	$f_i:\mR^{n_i}\to\mR$, constraint function $h_i:\mR^{n_i}\to\mR^q$, $A_i\in\mR^{p\times n_i}$, $b_i\in\mR^p$, and local constraint subset $\Omega_i\subseteq\mR^{n_i}$. The local information $\mathcal{F}_i=\left\{f_i, A_i,b_i,h_i,\Omega_i\right\}$ can only be known by agent $i$, but none of the other agents have access to this information. We aim to develop a distributed algorithm to solve \eqref{P}, which means that all agents can only access and process local data, and communicate only with their immediate neighbors. Through exchanging information, they cooperate to find an optimal solution to \eqref{P}.
	
	Various distributed algorithms have been proposed to solve optimization problems 
	with coupled constraints. The coupled linear
	equality constrained optimization problems were widely studied at first and probably 
	inspired by distributed resource allocation problems (e.g., \cite{Nedic18, Li20}). 
	The augmented Lagrangian method (ALM) and the alternating direction method of multipliers (ADMM) are classical and efficient algorithms for equality constrained convex optimization problems in centralized environments. These two
	algorithms are extended to solve distributed optimization problems with coupled linear constraints. To name a few, 
	Tracking-ADMM \cite{Falsone20}, Non-Ergodic Consensus-Based Primal-Dual Algorithm 
	\cite{Su22} (NECPD), Consensus-Based Distributed Augmented Lagrangian 
	Method \cite{Zhang18} (C-ADAL), Primal-Dual Consensus ADMM 
	\cite{Chang16} (PDC-ADMM). Among the above listed algorithms, Tracking-ADMM is 
	asymptotically convergent, while NECPD, C-ADAL, and PDC-ADMM all have
	$O(1/k)$ convergence rates, in terms of the objective residual and feasibility. The first-order optimality residual of Mirror-P-EXTRA proposed by \cite{Nedic18} decays to $0$ at a rate of $o(1/k)$.
	
	Existing methods for tackling coupled nonlinear inequality constraints can be 
	classified into two categories: primal decomposition and dual decomposition. The 
	primal decomposition method introduces artificial variables to transform the coupled 
	inequality constraints into a more manageable form. For instance, the authors of 
	\cite{Andrea21} transformed the coupled inequality constrained convex optimization 
	into a bi-level convex programming problem. Their proposed distributed primal-dual 
	algorithm has a convergence rate of $O(1/\log(k))$ for the objective residuals. 
	Similarly, the authors of \cite{WuXY22} employed artificial variables to decouple the 
	inequality constraints, at the cost of adding a coupled linear equality constraint working on the auxiliary variables. Their IPLUX algorithm achieves a convergence rate of $O(1/k)$ in terms of feasibility and optimality. On the other hand, the dual decomposition method utilizes Lagrange duality to establish a consensus-based dual 
	problem. Various distributed consensus-based algorithms can then be integrated to 
	solve the dual problem. For example, the distributed dual subgradient algorithm proposed by \cite{Falsone17,Liang21} has an ergodic $O(1/k)$ convergence rate. The authors of \cite{Liu20} integrated the dual gradient tracking method (see \cite{DIGing}) and a primal recovery technique to develop a primal-dual algorithm with objective residuals decaying to $0$ at a rate of $O(1/k)$. The work of \cite{Liang20} focused on smooth convex optimization problems and proposed a distributed algorithm that performs two successive gradient projection steps in each round of iteration for a min-max problem formulated by the dual decomposition. They showed that their algorithm with a small fixed stepsize asymptotically converges to a saddle point. Recently, the authors of \cite{AL-Tracking} proposed an augmented Lagrangian tracking (AL-tracking) method, and they demonstrated that any cluster of the primal decision sequences generated by the AL-tracking algorithm is an optimal solution.
	
		\begin{table}[t]
		\centering
		\begin{threeparttable}
			\caption{An overview of some existing distributed coupled constrained optimization algorithms.\\ \\
				\hspace*{0.5em}
				The full names of the abbreviations in this table are 
				listed there: C (convex), SC (strongly convex), NSO (non-smooth 
				optimization), E (equality),
				I (inequality (probably nonlinear)), OR (objective residual),
				FR (feasibility residual), 
				FOOR (first-order optimality residual).
			}
			\label{table_1}
			\begin{small}
				\begin{tabular}{|c|c|m{2.8cm}<{\centering}|m{0.8cm}<{\centering}|c|c|c|c|c|}
					\hline
					\multirow{2}{2cm}{Algorithms}
					&\multirow{2}{*}{NSO}&\multirow{2}{*}
					{\makecell[l]{Communications}}&
					\multicolumn{2}{c|}{\makecell[l]{Coupled\\constraints}}&
					\multirow{2}{2.0cm}{Convergence of $\{\x^k\}_{k=0}^\infty$
						\tnote{$\dagger$}}&\multicolumn{3}{c|}{\makecell{Convergence rate}}\\
					\cline{4-5}\cline{7-9}
					~&~&~& E & I & ~ &\multicolumn{2}{c|}{C} & SC  \\ 
					\hline
					\makecell{NECPD\\ \cite{Su22}} & $\surd$ 
					& $(d-1)$\tnote{$\dagger\dagger$} &
					$\surd$ & $\times$ & {\bf---} &\makecell{OR \&  FR}  
					& $O(1/k)$ &{\bf---} \\
					\hline
					\makecell{Tracking-\\ADMM\\ \cite{Falsone20}} & $\surd$ & 1
					& $\surd$ & $\times$& $\times$&\multicolumn{3}{c|}{{\bf---}} \\
					\hline
					\makecell{AL-Tracking \\ \cite{AL-Tracking}} & $\surd$ & 1 &
					$\surd$ & $\surd$ & $\times$ & \multicolumn{3}{c|}{{\bf---}}\\
					\hline
					\makecell{PDC-\\ADMM\\ \cite{Chang16}} & $\surd$ & 2&
					$\surd$ & $\times$ & $\surd$ & \makecell{OR \&  FR}  
					& $O(1/k)$ &{\bf---} \\
					\hline
					\cite{Liang20}  &$\times$ & 2
					&$\surd$& $\surd$& $\surd$&
					\multicolumn{3}{c|}{{\bf---}} \\
					\hline
					\makecell{IPLUX\\ \cite{WuXY22}} &$\surd$& 2 & $\surd$&
					$\surd$& {\bf---} & \makecell{OR \& FR} & $O(1/k)$ &{\bf---}\\
					\hline
					\cite{Li20} \tnote{$\dagger\dagger\dagger$} & $\times$ 
					& 1 & $\surd$ & $\times$ & $\surd$ 
					& \multicolumn{2}{c|}{{\bf ---}} & linear \\
					\hline
					\makecell{Mirror-\\P-EXTRA\\ \cite{Nedic18}} & $\surd$ & 1 &
					$\surd$ & $\times$ & $\surd$ & FOOR & $o(1/k)$ & linear \\
					\hline
					\makecell{This\\ paper} & $\surd$ & 1 &$\surd$&$\surd$&
					$\surd$ & FOOR &  $o(1/k)$ & linear \\
					\hline
				\end{tabular}
			\end{small}
			\begin{tablenotes}
				\footnotesize
				\item[$\dagger$] $\{\x^k\}_{k=0}^\infty$ is the decision
				variable sequence of agents, generated by each algorithm in this table.
				
				\item[$\dagger\dagger$] $d$ is the degree of the minimal 
				polynomial of the adjacency matrix $W$.
				
				\item [$\dagger\dagger\dagger$] The algorithm of \cite{Li20}
				can be applied to time-varying directed graphs, and other algorithms in this table employ a fixed undirected graph. 
			\end{tablenotes}
		\end{threeparttable}
	\end{table}
	
	In this paper, we will apply the dual decomposition method twice successively to handle coupled constraints and the consensus dual problem, to construct a Karush-Kuhn-Tucker (KKT) system of \eqref{P} that is separable with respect to the primal decision variables and the Lagrange multipliers respectively. This KKT system would be treated as an inclusion problem of a maximal monotone operator that can employ the proximal point method -- an efficient algorithm in a centralized environment, to solve it. 
	By strategically integrating ideas from the variable metric proximal point method and the prediction-correction framework presented by \cite{He22}, 
	we propose a decentralized proximal method of multipliers (DPMM) for this monotone inclusion problem, which is also a distributed algorithm for \eqref{P}.
	
	In the DPMM algorithm, each agent controls three variables: the primal decision variable, a local estimate (copy) of the Lagrange multiplier corresponding to the coupled constraints, and an introduced auxiliary variable. The local estimates of the Lagrange multiplier enable the algorithm to be executed in a distributed fashion, while the auxiliary variables push agents' local estimates to be consensual. During each round of iteration, all agents communicate with neighbors for only one round, exchanging only the values of agents' local estimates of the Lagrange multiplier. The contribution of this work to the research of distributed constrained optimization is four-fold. 
	\begin{enumerate}[\hspace{1em}$\bullet$]
		\item The decentralized proximal method of multipliers proposed in this article solves a challenging family of constrained convex optimization problems, where constraints (including equality and nonlinear inequalities) are coupled, the objective and constraint functions can be general closed convex (probably 
		non-smooth), and the local constraint sets of agents are closed convex and not required to be bounded. The assumptions of a bounded feasible region and smooth objective and constraint functions, either of which is an indispensable condition for many distributed constrained optimization algorithms to guarantee convergence, e.g., \cite{Zhang18, Chang16, 
			Liu20, Falsone20, Falsone17, Su22, Andrea21, Liang20}; 
		\item DPMM is a primal-dual convergent method. The sequences of primal decision variables and Lagrange multipliers generated by DPMM converge to a primal-optimal and dual-optimal solution, respectively. In contrast, some algorithms in the literature have only dual convergence, while the convergence of their generated primal decision sequence is either
		unknown (e.g., \cite{Liu20, Andrea21, Zhang18, WuXY22})
		or each cluster is a primal optimal solution
		(e.g., \cite{Falsone17, Falsone20, AL-Tracking}). In addition, as shown in Table {\bf\ref{table_1}}, DPMM has a lower communication burden compared with some other existing algorithms. 
		
		\item In this paper, an inexact version of the decentralized proximal method of multipliers (inexact-DPMM) is presented to relax the restriction that the minimization subproblems are solved exactly by agents at each iteration (e.g., \cite{WuXY22, Falsone20, Falsone17, Su22, Zhang18, Liang20, Liu20}). Inexact-DPMM is still convergent to an optimal solution of \eqref{P}, and if the precision of minimizing subproblem geometrically decays to $0$, its first-order optimality residual then converges to zero at a rate of $o(1/k)$.
		
		\item Under a structural assumption that includes strong convexity
		of the objective function, we demonstrate that inexact-DPMM has a linear convergence rate. Many other distributed coupled constrained optimization algorithms in literature do not have results about the linear convergence rate, for a survey, see Table {\bf\ref{table_1}}.
	\end{enumerate}
	
	The rest of the paper is organized as follows: In Section \ref{sec_2}, we state some standard assumptions about the network model of agents and the considered optimization problems, respectively. In Section \ref{sec_3}, we utilize the dual decomposition method to reformulate the KKT system of the problem \eqref{P} and present the distributed proximal method of multipliers, including an inexact version. In Section \ref{sec_4}, we analyze the convergence and rate of inexact-DPMM. In Section \ref{sec_5}, we numerically simulate DPMM for solving two examples and compare it with some algorithms in Table {\bf\ref{table_1}} to illustrate its effectiveness. In section \ref{sec_6}, we make a conclusion. 
	
	Notation: Let $\mR^n$ be the $n$-dimensional Euclidean space with the inner 
	product $\langle\cdot,\cdot\rangle$ and norm $\Vert\cdot\Vert$. We use $\1$ to 
	denote an $m$-dimensional vector whose entries are all ones, and $I_p$ to denote the identity matrix with dimension $p$. The notation ${\rm col}(\x,\y)$ means a column 
	vector with $\x$ and $\y$ as its components. Suppose $\Omega$ is a nonempty subset of $\mR^n$. We write ${\rm int}(\Omega)$ as its interior. $\delta_{\Omega}(\x)$ is the indicator function, which is $0$ if $\x\in\Omega$, and $+\infty$ otherwise. 
	The symbol $\mP_\Omega(\x)$ stands for the projection of $\x$ onto $\Omega$. If $\K$ 
	is a closed convex cone in $\mR^n$, we denote its polar by $\K^\circ$. Let $U$ be an 
	$n\times m$ matrix. Its transpose is denoted by $U^\top$. The notation $U\succ0$ 
	means that the matrix $U$ is symmetric positive definite. The null space of $U$, i.e., the set of all matrices $X\in\mR^{m\times r}$ such that $UX=0$, is denoted by ${\rm null}\{U\}$. The notation ${\rm diag}(A,B)$ defines a block diagonal matrix with $A$ and $B$ as its diagonal blocks. We use $\otimes$ to denote the Kronecker product. Suppose $\mH$ is a symmetric positive definite matrix, we define the $\mH$-matrix norm of a vector $\x$ as $\Vert\x\Vert_\mH=\sqrt{\langle\x,\mH\x\rangle}$. 
	$\lambda_{\max}(\mH)$ and $\lambda_{\min}(\mH)$ denote the maximum and minimum eigenvalue of the matrix $\mH$, respectively.
	To avoid confusion, we use the notation $\x^k$ for the value of $\x$ at the $k$-th iteration, and $(\x)^s$ for the $s$-th power of $\x$.
	
	\section{Network Moldel and Assumptions}\label{sec_2}
	
	Consider a multi-agent network consisting of $m$ agents, we model it as an undirected graph $\G=(\V,\E)$, where $\V=\{1,2,\dots,m\}$ is the sets of nodes (i.e., agents), and $\V$ is the set of edges. An edge $(i,j)\in\E$ if and only if agent $i$ and agent $j$ can communicate and exchange information with each other. We define the index subset of neighbors of agent $i$ as $\N_i=\{j\in\V: (i,j)\in\E\}$, for all $i\in\V$. The adjacency matrix $W$ of the graph $\G$ satisfies $W_{ij}=1$, if $(i,j)\in\E$, and $W_{ij}=0$ otherwise. Let $\text{\L}_{\G}$ denote the Laplacian matrix which is compatiable with $\G$, i.e., $\text{\L}_{\G}=D-W$, where $D$ is a diagonal matrix with the entry $D_{ii}$ being the degree of agent $i$. A few facts about
	$\text{\L}_{\G}$ are that $\text{\L}_{\G}$ is symmetric and positive semidefinite and
	satisfies ${\rm null}\{\text{\L}_{\G}\}=\{t\1: t\in\mR\}$.
	
	In this article, we use a matrix $\text{\L}=[\text{\L}_{ij}]$ that can have more choices, which plays a role close to or the same as $\text{\L}_{\G}$.
	\begin{ass}[Graph connectivity]\label{ass_1}
		The undirected graph $\G$ is connected and a matrix {\rm\L} is compatible
		with $\G$. Furthermore, $\text{\rm\L}=U^\top U$ for some full row-rank matrix $U$
		and $U\1=0$.
	\end{ass}
	The connectivity assumption of $\G$ is standard for distributed optimization. 
	Reference \cite{PG-EXTRA} uses $U=\sqrt{\text{\L}}$, while we use the decomposition
	$\text{\L}=U^\top U$. The null properties of \L \ and $U$ in Assumption 
	{\bf\ref{ass_1}} imply that ${\rm null}\{\text{\L}\}={\rm null}\{U\}$. Let 
	\begin{equation}\label{UL}
		\textbf{\L}\triangleq\text{\L}\otimes I_{p+q}\ , \quad \U\triangleq U\otimes I_{p+q}\ .
	\end{equation}
	Thus we have 
	\begin{equation}\label{consensus}
		\y_1=\y_2=\cdots=\y_m\Longleftrightarrow 
		\textbf{\L}\Y=0\Longleftrightarrow \U\Y=0,
	\end{equation}
	where $\Y=\mathrm{col}\left(\y_1,\y_2,\dots,\y_m\right)\in\mR^{m(p+q)}$.
	Relation \eqref{consensus} is a common technique for distributed algorithms that employ the dual decomposition method, and it plays a key role in extending many centralized algorithms to distributed settings.
	
	\begin{rem}\label{rem_1}
		The matrix {\rm\L} can be chosen in several different ways:
		\begin{enumerate}[{\bf (a)}]
			\item Since the Laplacian matrix $\text{\rm\L}_{\G}$ satisfies Assumption {\bf\ref{ass_1}}, we can choose $\text{\rm\L}=\text{\rm\L}_{\G}$. In this case, each agent needs to know the number of its neighbors \textit{(its degree)} and 
			{\rm\L} can be constructed without any communication among the agents.
			
			\item We can choose $\text{\rm\L}=(I-W)/\nu$ {\rm\cite{EXTRA,Xiao04,Sayed14}},
			where $\nu>0$ is a scaling parameter, and $W$ is a symmetric doubly stochastic matrix that is compatible with $\G$. This matrix can be constructed using some
			local strategies such as the Metropolis-Hasting rule {\rm\cite{Basar16}}.
		\end{enumerate}
	\end{rem}
	\begin{ass}[Convexity and existence of optimal solution]\label{ass_2}
		\hspace{1em}
		\begin{enumerate}[{\bf(a)}]
			\item \eqref{P} is a convex problem, i.e., each $f_i$ and $h_i$ 
			\textit{(vector-valued)} are general proper closed convex functions \textit{(probably non-smooth)}, and each $\Omega_i$ is an empty closed convex subset in $\mR^{n_i}$, for all $i=1,2,\dots,m$.
			\item The optimal solution set $\mathcal{X}^*$ of \eqref{P} is nonempty.
		\end{enumerate}
	\end{ass}
	Note that we do not require the boundedness of the local constraint subsets
	$\Omega_i$, $i=1,2,\dots,m$. However, many distributed constrained optimization algorithms in the literature rely on the boundedness (or compactness) of $\Omega_i$'s to bound the (sub)gradients $\nabla f_i$ (or $\partial f_i$) and the primal decision sequence $\{\x^k\}_{k\ge0}$ that they generate.
	
	Let $\K^\circ=\mR^p\times\mR^q_+$, the Lagrange dual problem of \eqref{P} is then
	formulated as
	\begin{equation}\tag{\bf D}\label{D}
		\max_{\y\in{\mR^{p+q}}}\varphi(\y)\triangleq\sum_{i=1}^m\varphi_i(\y),
	\end{equation}
	where 
	\begin{subequations}\label{Fun_Lag}
		\begin{align}
			\varphi_i(\y)&=\inf_{\x_i}\ell_i(\x_i,\y),\\
			\ell_i(\x_i,\y)&=f_i(\x_i)+\y^\top
			G_i(\x_i)+\delta_{\Omega_i}(\x_i)-\delta_{\K^\circ}(\y),\\
			\ell(\x,\y)&=\sum_{i=1}^m\ell_i(\x_i,\y).
		\end{align}
	\end{subequations}
    For constructing a primal-dual method, a basic assumption is that strong duality holds, which is guaranteed by the following standard assumption.
	\begin{ass}[Slater's condition]\label{ass_3}
		There exists a point $\bar{\x}=\mathrm{col}(\bar{\x}_1,\bar{\x}_2,\dots,
		\bar{\x}_m)$ such that $\bar{\x}_i\in\mathrm{int}\left(\Omega_i\right)$,
		$\forall\ i=1,2,\dots,m$, and $\sum_{i=1}^m\left(A_i\bar{\x}_i-b_i\right)=0,\ 
		\sum_{i=1}^mh_i\left(\bar{\x}_i\right)<0$.
	\end{ass}
	When using the dual decomposition to decouple the constraints in \eqref{P}, a dual variable $\y$ is introduced, while it's still coupled in \eqref{D}. To decouple $\y$, we need the following separable Lagrange function
	\begin{equation}\label{eq_dis-Lagrange}
		\mL(\x,\Y)\triangleq\sum_{i=1}^m\ell_i(\x_i,\y_i), \quad \x_i\in\mR^{n_i},\ 
		\y_i\in\mR^{p+q},
	\end{equation}
	where $\Y=\mathrm{col}(\y_1,\y_2,\dots,\y_m)$, and $\y_i$ is a local estimate (copy) 
	of the Lagrange multiplier $\y$, owned by agent $i$, for all $i=1,2,,\dots,m$. 
	It is clear that if all the estimates $\y_i$'s are identical, i.e., $\Y=\1\otimes\y$ for some $\y$, then 
	$\mL(\x,\Y)=\ell(\x,\y)$. Applying the consensus relation \eqref{consensus}, we 
	obtain the following first-order optimality conditions for \eqref{P}.
	\begin{lem}[First-order optimality conditions for \eqref{P}]\label{LEM_1}
		Let Assumptions {\bf\ref{ass_1}, \ref{ass_2}}, and {\bf\ref{ass_3}} hold. Then,	
		$\x$ and $\y$ are optimal solutions to \eqref{P} and \eqref{D} respectively,
		if and only if there exist 
		$\Z=\mathrm{col}\left(\z_1,\z_2,\dots,\z_m\right)\in\mR^{m(p+q)}$,
		and $\Y=\1\otimes\y$, such that 
		\begin{equation}\label{FOOC}
			\mathbf{0}\in\Phi\left(\x,\Y,\Z\right)
			\triangleq\left(
			\begin{array}{c}
				\partial_\x\mL(\x,\Y)\\
				-\partial_\Y\mL(\x,\Y)+\U^\top\Z\\
				-\U\Y
			\end{array}\right),
		\end{equation}
		where $\U$ is the matrix defined by \eqref{UL}.
	\end{lem}
	\begin{proof}
		From the Lagrange function $\ell(\x,\y)$, the first-order optimality
		conditions of \eqref{P} can be formulated as
		\begin{equation}\label{FOOC_P}
			\mathbf{0}\in\partial\ell(\x,\y)=\left(
			\begin{array}{c}
				\big(\partial_{\x_i}\mathit{l}_i(\x_i,\y)\big)_{i=1}^m\\
				-\sum_{i=1}^{m}\partial_\y\mathit{l}_i(\x_i,\y)
			\end{array}\right).
		\end{equation}
		We then need to show the equivalence between \eqref{FOOC_P} and \eqref{FOOC}.
		
		Suppose that $0\in\Phi(\x,\Y,\Z)$, then $\U\Y=0$ implies that there exists
		$\y$ such that $\Y=\1\otimes\y$, which follows $\mL(\x,\Y)=\mathit{l}(\x,\y)$.
		Hence it holds that $0\in\partial_\x\mathit{l}(\x,\y)$. In view of $\U\1=0$,
		multiplying $\1^\top$ by $0\in\left(-\partial_\Y\mL(\x,\Y)+\U^\top\Z\right)$
		yields
		\begin{equation*}
			0\in-\1^\top\partial_\Y\mL(\x,\Y)=
			-\sum_{i=1}^m\partial_{\y_i}\mathit{l}_i(\x_i,\y_i)
			=-\sum_{i=1}^m\partial_\y\mathit{l}_i(\x_i,\y).
		\end{equation*}

		Conversely, assume that $0\in\partial\mathit{l}(\x,\y)$.
		Let $\Y=\1\otimes\y$, it follows that $-\U\Y=0$ and 
		$\mL(\x,\Y)=\mathit{l}(\x,\y)$, thus $0\in\partial_\x\mL(\x,\Y)$
		and $0\in-\1^\top\partial_\Y\mL(\x,\Y)$, i.e., there exists  $V\in\partial_\Y\mL(\x,\Y)$ such that $\1^\top V =0$. We then have
		\begin{equation*}
			V\in\mathrm{span}\{\1\}^\perp=\mathrm{null}\{\U\}^\perp
			=\mathrm{span}\{\U^\top\},
		\end{equation*}
		which implies that $V=U^\top \Z$ for some $\Z$. Combining with 
		$V\in\partial_\Y\mL(\x,\Y)$, it holds that 
		$0\in-\partial_\Y\mL(\x,\Y)+\U^\top\Z$.
	\end{proof}
	Separable Lagrange function $\mL(\x,\Y)$ enables us to decompose \eqref{P} so that it can be processed in parallel by agents. Simultaneously, the role of variable $\Z$ is to ensure that all the estimates $\y_i$'s of the Lagrange multiplier $\y$ reach a consensus among agents.
	
	\section{Decentralized Proxmal Method of Multipliers}\label{sec_3}
	
	In Lemma {\bf\ref{LEM_1}}, we introduce an operator $\Phi$ to characterize the first-order optimality conditions of \eqref{P}. The operator $\Phi$ is maximal monotone by Minty’s theorem \cite[Theorem 21.1]{convex17}. To find a solution
	to this monotone inclusion problem of the operator $\Phi$, 
	the authors of References \cite{Corman14,Eckstein92,Tao18} developed the 
	following generalized proximal point algorithm
	\begin{subequations}\label{g_PPA}
		\begin{align}
			\widehat{\xi}^k&=(I+s\Phi)^{-1}(\xi^k),\\
			\xi^{k+1}&=\xi^k-r\left(\xi^k-\widehat{\xi}^k\right),
		\end{align}
	\end{subequations}
	where $s>0$ and $r\in(0,2)$. This algorithm has an optimal linear convergence rate,
	as shown in \cite{Tao18}. Althrough the algorithm \eqref{g_PPA} can theoretically find a solution $\xi^*$ satisfying $0\in\Phi(\xi^*)$, the computation cost of the proximal operator $(I+s\Phi)^{-1}$ is very expensive. We aim to develop a decentralized algorithm that decomposes the primal problem \eqref{P} into smaller subproblems and is then solved by a group of agents through communication and parallel computation.
	
	On the other hand, the authors of \cite{Ma20, He20} incorporated ideas from the unified prediction-correction framework presented by \cite{He22} and ADMM (or ALM), which is essentially a proximal point method for multipliers (see \cite{Rockafellar76}), to develop new algorithms with larger stepsizes or domains of convergence. Generally, algorithms with larger step-sizes often imply faster convergence. Under these observations, we attempt to integrate the generalized proximal point algorithm with the prediction-correction framework to propose the following decentralized proximal method of multipliers.
	To simplify our presentation, let us define the quantities:
	\begin{align*}
		\xi^k\triangleq\left(
		\begin{array}{c}
			\x^k\\
			\Y^k\\
			\Z^k
		\end{array}\right),\
		Q\triangleq\left(
		\begin{array}{ccc}
			\Upsilon^{-1} & 0 & 0\\
			0   & \Gamma^{-1} & -\U^\top\\
			0   &  0   & \frac{1}{\beta}I_{m(p+q)}
		\end{array}\right),\
		M\triangleq\left(
		\begin{array}{ccc}
			\Theta & 0 & 0\\
			0  & I_{m(p+q)} & -\Gamma\U^\top\\
			0   &   0  & I_{m(p+q)}
		\end{array}\right),
	\end{align*}
	where $\x^k=\mathrm{col}(\x_i^k)_{i=1}^m$, $\Y^k=\mathrm{col}(\y_i^k)_{i=1}^m$, $\Z^k=\mathrm{col}(\z_i^k)_{i=1}^m$, $\beta>0$, and
	\begin{align*}
		\Theta&=\mathrm{diag}\left(\theta_1I_{n_1},\theta_2I_{n_2},\dots,
		\theta_mI_{n_m}\right), \theta_i>0,\ \forall \ i\in\V,\\
		\Upsilon&=\mathrm{diag}\left(\alpha_1I_{n_1},\alpha_2I_{n_2},
		\dots,\alpha_mI_{n_m}\right), \alpha_i>0,\ \forall \ i\in\V, \\
		\Gamma&=\mathrm{diag}\left(\gamma_1I_{p+q},\gamma_2I_{p+q},
		\dots,\gamma_mI_{p+q}\right),\gamma_i>0,\ \forall \ i\in\V.
	\end{align*} 
    Applying the proximal point algorithm in \cite{Rockafellar76} to the generalized equation \eqref{FOOC}, we obtain the decentralized proximal method of multipliers
    (DPMM) described in terms of operators as follows: at iteration $k$, given $\xi^k$,
    $\xi^{k+1}$ is generated by
	\begin{subequations}\label{pc_PPA}
		\begin{align}
			\widehat{\xi}^k&=\left(Q+\Phi\right)^{-1}\left(Q\xi^k\right),
			\label{pc_PPA_a}\\
			\xi^{k+1}&=\xi^k-M\left(\xi^k-\widehat{\xi}^k\right),\label{pc_PPA_b}
		\end{align}
	\end{subequations}
	where $\widehat{\xi^k}=\mathrm{col}\left(\widehat{\x}^k,
	\widehat{\Y}^k,\widehat{\Z}^k\right)$.
	
	Different from the P-EXTRA algorithm \cite{PG-EXTRA}, it is a variable metric proximal point method with a symmetric positive definite metric matrix. Our method 
	\eqref{pc_PPA} employs two asymmetric invertible matrices $Q$ and $M$. $Q$ is used to obtain a prediction of the proximal point, and $M$ to correct it.
	
	Formulas \eqref{pc_PPA} are written in an operator form, which will be used in the
	convergence analysis. The following proposition presents its variables updating rules in practical computation.
	\begin{Prop}\label{PROP_4}
		Given initial points $\xi^0=\mathrm{col}\left(\x^0,\Y^0,\Z^0\right)$, let
		$\xi^k=\mathrm{col}\left(\x^k,\Y^k,\Z^k\right)$, 
		$\widehat{\xi}^k=\mathrm{col}\left(\widehat{\x}^k,\widehat{\Y}^k,
		\widehat{\Z}^k\right)$, $\Lambda^k=\U^\top\Z^k$, and $\widehat{\Lambda}^k=\U^\top\widehat{\Z}^k$ 
		 for all $k\ge0$. The iterative schemes \eqref{pc_PPA} are equivalent to
		\begin{equation}\tag{DPMM}
      	\left\{
			\begin{aligned}
				\widehat{\x}^k&=\mathop{\arg\min}\limits_{\x\in\mR^n}
				\phi^k\left(\x,\Y^k-\Gamma\Lambda^k\right),\\
				\widehat{\Y}^k&=\mP_{\K_m^\circ}\left(\Y^k-
				\Gamma\Lambda^k+\Gamma G_d\left(\widehat{\x}^k\right)\right),\\
				\x^{k+1}&=(I_n-\Theta)\x^k+\Theta \ \widehat{\x}^k,\\
				\Lambda^{k+1}&=\Lambda^k+\beta\text{\bf\L}\widehat{\Y}^k,\\
				\Y^{k+1}&=\widehat{\Y}^k+\Gamma\left(\Lambda^k-\Lambda^{k+1}\right),
			\end{aligned}
			\right.
		\end{equation}
	where $\phi^k(\x,\Y)$ is the separable augmented Lagrangian with a quadratic proximal term, i.e. 
	\begin{subequations}\label{fun_p_ALM}
		\begin{align}
			\phi^k(\x,\Y)&\triangleq\sum_{i=1}^m\phi_i^k(\x_i,\y_i),\\
			\phi_i^k(\x_i,\y_i)&\triangleq f_i(\x_i)+\frac{1}{2\gamma_i}\bigg(
			\left\Vert\mP_{\K^\circ}\left(\y_i+\gamma_i G_i(\x_i)\right)\right\Vert^2-
			\Vert\y_i\Vert^2\bigg)\notag\\
			&\hspace{3.7em}+\frac{1}{2\alpha_i}\left\Vert\x_i-\x_i^k\right\Vert^2
			+\delta_{\Omega_i}(\x_i),
		\end{align}
	\end{subequations}
    $G_d(\x)=\mathrm{col}\big(G_1(\x_1),G_2(\x_2),\dots,G_m(\x_m)\big)$,
    $\mP_{\K_m^\circ}=\underbrace{\mP_{\K^\circ}\times
    \mP_{\K^\circ}\times\cdots\times\mP_{\K^\circ}}_{m}$, and the projection is given by
    $\mP_{\K^\circ}(x_1,x_2)=(x_1,\max(x_2,0))$, for any $(x_1,x_2)\in\mR^p\times\mR^q$.
	\end{Prop}
	\begin{proof}
		Expanding \eqref{pc_PPA_a} yields
		\begin{align*}
			\Upsilon^{-1}\left(\x^k-\widehat{\x}^k\right)
			&\in\partial_\x\mL(\widehat{\x}^k,\widehat{\Y}^k),\\
			\Gamma^{-1}\left(\left(\Y^k-\Gamma\U^\top\Z^k\right)-\widehat{\Y}^k\right)
			&\in-\partial_\Y\mL(\widehat{\x}^k,\widehat{\Y}^k),\\
			\widehat{\Z}^k &=\Z^k+\beta\U\widehat{\Y}^k.
		\end{align*}
	Note that $\text{\bf\L}=\U^\top\U$, multiplying $\U^\top$ by 
	$\widehat{\Z}^k =\Z^k+\beta\U\widehat{\Y}^k$, it follows that
		\begin{subequations}\label{ppa_L}
			\begin{align}
				\Upsilon^{-1}\left(\x^k-\widehat{\x}^k\right)
				&\in\partial_\x\mL(\widehat{\x}^k,\widehat{\Y}^k),\label{ppa_L1}\\
				\Gamma^{-1}\left(\left(\Y^k-\Gamma\Lambda^k\right)-\widehat{\Y}^k\right)
				&\in-\partial_\Y\mL(\widehat{\x}^k,\widehat{\Y}^k),\label{ppa_L2}\\
				\widehat{\Lambda}^k &=\Lambda^k+\beta\text{\bf\L}\widehat{\Y}^k.\label{ppa_L3}
			\end{align}
		\end{subequations}
		It is clear that \eqref{ppa_L1}-\eqref{ppa_L2} are the proximal point method applied to the subdifferential operator $\partial\mL$, i.e.,
		\begin{align*}
			\left[
			\begin{array}{c}
				\widehat{\x}^k \\ \widehat{\Y}^k
			\end{array}\right]
			=
			\bigg(S+\partial\mL\bigg)^{-1}\left(
			S\left[
			\begin{array}{c}
				\x^k \\ \Y^k-\Gamma\Lambda^k		
			\end{array}\right]\right),
		\end{align*}
		where $S=\mathrm{diag}(\Upsilon^{-1},\Gamma^{-1})$. 
		Recall that $\ell_i(\x_i,\y_i)=f_i(\x_i)+\y_i^\top G_i(\x_i)+\delta_{\Omega_i}(\x_i)-\delta_{\K^\circ}(\y_i)$, it 
		can be viewed as the Lagrange function associated with the following local
		constrained optimization problem:
		\begin{align*}
			\min_{\x_i\in\mR^{n_i}} f_i(\x_i) \quad \mathrm{s.t.} \
		     \ G_i(\x_i)\in\K,\ \x_i\in\Omega_i.
		\end{align*}
		Moreover, it's clear that $\phi_i^k(\x_i,\y_i)$ is the augmented Lagrangian with a quadratic proximal term. $\phi_i^k(\x_i,\y_i)$ is often employed in the proximal
	    method of multipliers. A fact is that the proximal method of multipliers is equivalent to the proximal point algorithm applied to the subdifferential operator of the Lagrange function (see the minimax application of the proximal point method in \cite{Rockafellar76}). Consequently, by the separability of the functions 
	    $\mL(\x,\Y)$ and $\phi^k(\x,\Y)$, it holds that
		\begin{equation}\label{p_ALM}
			\begin{aligned}
				\left.
				\begin{array}{c}
					\widehat{\x}^k=\mathop{\arg\min}\limits_{\x\in\mR^n}
					\phi^k\left(\x,\Y^k-\Gamma\Lambda^k\right)\\
					\widehat{\Y}^k=\mP_{\K_m^\circ}\left(\Y^k-
					\Gamma\Lambda^k+\Gamma G_d\left(\widehat{\x}^k\right)\right)
				\end{array}
				\right\}
				\Longleftrightarrow \mathrm{\eqref{ppa_L1}}\ \text{and} \ \mathrm{\eqref{ppa_L2}}.
			\end{aligned}
		\end{equation}
	Using the definitions of $\Lambda^k$ and $\widehat{\Lambda}^k$,
	the correction step \eqref{pc_PPA_b} can be expanded as 
	\begin{subequations}\label{itera_c}
		\begin{align}
			\x^{k+1}&=(I-\Theta)\x^k+\Theta\ \widehat{\x}^k,\\
			\Y^{k+1} &=\widehat{\Y}^k+\Gamma\left(\Lambda^k-\widehat{\Lambda}^k\right),\\
			\Lambda^{k+1}&=\widehat{\Lambda}^k,
		\end{align}
	\end{subequations}
    which completes the proof.
	\end{proof}

    Denote $\Lambda^k=\mathrm{col}\left(\lambda_1^k,\lambda_2^k,\dots,\lambda_m^k\right)$,
    for simplicity, initialize with $\lambda_i^0=0$, $\forall\ i\in\V$, i.e., 
    $\Lambda^0=0$. Algorithm {\bf\ref{alg_1}} is derived by expanding the DPMM algorithm in coordinate blocks, and it shows how agents cooperate together to minimize the problem 
    \eqref{P} through communication with their neighbors and local computations.
	\IncMargin{1em}
	\begin{algorithm}[htbp]
		\caption{Inexact decentralized proximal method of multipliers (inexact-DPMM)}
		\label{alg_1}
		\SetKwInOut{Input}{\textbf{Parameters}}
		\SetKwInOut{Output}{\textbf{Initialization}}
		\Input{
			Each agent $i$ chooses its parameters 
			$\{\theta_i, \alpha_i, \gamma_i\}>0$, $\forall\ i\in\V$. 
			All agents agree on a common parameter $\beta>0$, and 
			construct matrix {\rm\L}\textup{\ (see Remark {\bf\ref{rem_1}})}.
		}
		\Output{
			Each agent $i$ chooses arbitrary $\x_i^0\in\Omega_i$, 
			$\y_i^0\in\K^\circ, \lambda_i^0=0$. set $k=0$.
		}
		\BlankLine
		\While{do not satisfy some stopping criterion}{
			\For{$i=1,2,\dots,m$\textup{\ (in parallel)}}{
				Agent $i$ select a precision $\varepsilon_i^k\ge0$, and compute 
				$\widehat{\x}_i^k\approx\mathop{\arg\min}\limits_{\x_i\in\mR^{n_i}}\phi_i^k
				(\x_i,\y_i^k-\gamma_i\lambda_i^k)$ using the following criterion:
				\begin{equation}\label{subproblem}
					\exists\ v_i^k\in \partial_{\x_i}\phi_i^k
					\left(\widehat{\x}_i^k,\y_i^k-\gamma_i\lambda_i^k\right) \ \text{such that}\
					\left\Vert v_i^k\right\Vert\le\varepsilon_i^k.
				\end{equation}\\
				\vspace{0.5em}
				Agent $i$ updates $\widehat{\y}_i^k=\mP_{\K^\circ}
				\left(\y_i^k-\gamma_i\lambda_i^k+
				\gamma_iG_i\left(\widehat{\x}_i^k\right)\right)$.\\
				\vspace{0.5em}
				Agent $i$ sends $\widehat{\y}_i^k$ to each neighnor $l\in\N_i$, and
				receives $\widehat{\y}_j^k$ from each neighbor $j\in\N_i$. \\
				\vspace{0.5em}
				Agent $i$ updates $\x_i^{k+1}=(1-\theta_i)\x_i^k+
				\theta_i\ \widehat{\x}_i^k$.\\
				\vspace{0.5em}
				Agent $i$ updates $\lambda_i^{k+1}=\lambda_i^k+\beta
				\sum_{j\in\N_i\cup\{i\}}{\text{\rm\L}}_{ij}\widehat{\y}_j^k$.\\
				\vspace{0.5em}
				Agent $i$ updates $\y_i^{k+1}=\widehat{\y}_i^k+\gamma_i
				\left(\lambda_i^k-\lambda_i^{k+1}\right)$.\\
			}
			$k\leftarrow k+1$, go to step 2.
		}
	\end{algorithm}
	
	Inexact-DPMM is more practical because it allows agents to use an inexact criterion \eqref{subproblem} to solve subproblems, which is computationally implementable for a wide variety of problems. Some distributed constrained optimization algorithms, such as dual subgradient algorithms \cite{Falsone17, Liang21},
	necessitate exact solutions of subproblems to obtain dual subgradients. However, if the subproblems cannot be solved exactly, only $\varepsilon$-subgradients are available, whether these algorithms still guarantee convergence is a significant question.
	
	In addition, DPMM has less communication burden. As shown in Algorithm {\bf\ref{alg_1}}, each agent only needs to communicate the value of
	$\widehat{\y}_i^k$ with neighbors once per iteration. In contrast, some alternative algorithms require multiple rounds of communication per iteration or exchange more than one variable value (cf. Table {\bf\ref{table_1}}).
	
	\section{Convergence Analysis}\label{sec_4}
	
	In this section, we analyze the proposed DPMM algorithm and state its convergence properties. According to the convergence conditions of the prediction-correction framework presented by \cite{He22}, the matrices $Q$ and $M$ are required to satisfy
	\begin{equation}
		\mH\triangleq QM^{-1}\succ0,\quad \D\triangleq\left(Q^\top+Q-M^\top\mH M\right)\succ0.
	\end{equation}
	\begin{Prop}\label{PROP_1}
		If the parameters $\theta_i\in(0,2)$, $\alpha_i>0$, $\gamma_i>0$, and
		$0<\gamma_i\beta<\frac{1}{\lambda_{\max}(\text{\rm\L})}, 
		\forall \ i=1,2,\dots,m$, the matrices $\mH$ and $\D$ satisfy that 
		$\mH\succcurlyeq \D\succ 0$.
	\end{Prop}
	\begin{proof} 
		By some calculations, one has
		\begin{align*}
			\mH=\left(
			\begin{array}{ccc}
				\Upsilon^{-1} & 0 & 0\\
				0   & \Gamma^{-1} & -\U^\top\\
				0   &  0   & \frac{1}{\beta}I_{m(p+q)}
			\end{array}\right)
			\left(
			\begin{array}{ccc}
				\Theta^{-1} & 0 & 0\\
				0   & I_{m(p+q)} & \Gamma\U^\top\\
				0   &   0  & I_{m(p+q)}
			\end{array}
			\right)
			=\left(
			\begin{array}{ccc}
				\Upsilon^{-1}\Theta^{-1} & 0 & 0\\
				0 & \Gamma^{-1} & 0 \\
				0 & 0     & \frac{1}{\beta}I_{m(p+q)}
			\end{array}\right),
		\end{align*} 
		\begin{align*}
			\D=Q^\top+Q-M^\top\mH M=Q^\top+Q-M^\top Q=\left(
			\begin{array}{ccc}
				(2I_n-\Theta)\Upsilon^{-1} & 0 & 0\\
				0 & \Gamma^{-1} & 0\\
				0 & 0 & \frac{1}{\beta}I_{m(p+q)}-\U\Gamma\U^\top \\ 
			\end{array}\right),
		\end{align*}
		\begin{align*}
			\mH-\D=\left(
			\begin{array}{ccc}
				(\Theta^{-1}+\Theta-2I_n)\Upsilon^{-1} & 0 & 0\\
				0 & 0 & 0\\
				0 & 0 & \U\Gamma\U^\top
			\end{array}
			\right).
		\end{align*}
	    In view of the basic inequality: $\frac{1}{\theta_i}+\theta_i\ge 2$ for any 
	    $\theta_i>0$, it holds that
	    $\left(\Theta^{-1}+\Theta-2I_n\right)\succcurlyeq0$.
	    if $\theta_i\in(0,2),\ \alpha_i>0,\ \gamma_i>0, \ \forall\ 
	    i=1,2,\dots,m$, the matrices $(2I_n-\Theta), \Upsilon^{-1}, \Theta^{-1}$, and
	    $\Gamma^{-1}$ are all symmetric positive definite, which implies that 
	    $\mH\succcurlyeq\D$ and $\mH\succ0$. Clearly,
		$\D\succ0$ if and only if $\frac{1}{\beta} I_{m(p+q)}-\U\Gamma\U^\top\succ0$.
		Using the Schur complement lemma \cite{Schur05} and $\text{\bf\L}=\U^\top\U$, 
		we have
		\begin{equation*}
			\D\succ0\Longleftrightarrow\frac{1}{\beta} I_{m(p+q)}-\U\Gamma\U^\top\succ0\Longleftrightarrow
			\Gamma^{-1}-\beta\U^\top\U=\Gamma^{-1}-\beta\text{\bf\L}\succ0.
		\end{equation*}
		The positive definiteness of matrix $\Gamma^{-1}-\beta\textbf{\L}$ implies
		that $0<\gamma_i\beta<\frac{1}{\lambda_{\max}(\text{\rm\L})}$, $\forall\ i\in\V$.
	\end{proof}
	\begin{rem}\label{rem_2}
		Note that the bound $0<\gamma_i\beta<\frac{1}{\lambda_{\max}(\text{\rm\L})}$
		does not necessarily imply that the selections of $\beta$ and $\gamma_i$  require any knowledge of the structure of the graph $\G$. For example, such a requirement can be avoided by using 
		$\text{\rm\L}=(I-W)/\nu$ \textit{(cf. Remark {\bf\ref{rem_1}})}. In this case,
		$0<\lambda_{\max}(\text{\rm\L})<2/\nu$, one may use $0<\gamma_i\beta\le\nu/2$ which is sufficient for $0<\gamma_i\beta<\frac{1}{\lambda_{\max}(\text{\rm\L})}$.
	\end{rem}
	The iterative schemes \eqref{pc_PPA} are essential for the convergence analysis. If agents use criterion \eqref{subproblem} to minimize 
	$\phi^k(\x,\Y^k-\Gamma\Lambda^k)$, formulas \eqref{pc_PPA_a} becomes an inexact proximal point method. The following lemma is a key insight for the inexact-DPMM 
	algorithm.
	\begin{lem}\label{LEM_2}
		Under Assumptions \textup{\bf\ref{ass_1}, \ref{ass_2}}, and 
		\textup{\bf\ref{ass_3}},
		let $\left\{(\x^k,\Y^k,\Lambda^k)\right\}_{k\ge0}$ and
		$\left\{(\widehat{\x}^k,\widehat{\Y}^k,\widehat{\Lambda}^k)\right\}_{k\ge0}$
		be any sequences generated by Algorithm {\rm\bf\ref{alg_1}}, where
		$\widehat{\Lambda}^k=\Lambda^{k+1}, \forall \ k\ge0$. Denote
		$\xi^k=\mathrm{col}\left(\x^k,\Y^k,\Z^k\right)$ and
		$\widehat{\xi}^k=\mathrm{col}\left(\widehat{\x}^k,\widehat{\Y}^k,
		\widehat{\Z}^k\right)$, where the sequences $\left\{Z^k\right\}_{k\ge0}$ and 
		$\left\{\widehat{Z}^k\right\}_{k\ge0}$ satisfy 
		$\Lambda^k=\U^\top\Z^k$ and $\widehat{\Lambda}^k=\U^\top\widehat{\Z}^k$,
		respectively. It holds that
		\begin{equation}\label{error_PPA}
			\begin{aligned}
				\widehat{\xi}^k&=\left(Q+\Phi\right)^{-1}\left(Q\xi^k+V^k\right),\\	
				\xi^{k+1}& = \xi^k-M\left(\xi^k-\widehat{\xi}^k\right),
			\end{aligned}
		\end{equation}
	\end{lem} 
	\noindent where $V^k=\mathrm{col}\left(\left(v_i^k\right)_{i=1}^m,0_{m(p+q)},0_{m(p+q)}\right)$ and $\left\Vert V^k\right\Vert\le\varepsilon^k\triangleq
	\left(\sum_{i=1}^m(\varepsilon_i^k)^2\right)^{1/2}$.
	\begin{proof}
			The proof for this lemma is basically the same as
		\cite[Proposition 8]{Rockafellar76}. Using the calculation rule of subdifferentials for $\phi^k(\x,\Y)$ and $\mL(\x,\Y)$, we obtain 
		\begin{equation*}
		  v_i^k\in \partial_{\x_i}\phi_i^k\left(\widehat{\x}_i^k,
		         \y_i^k-\gamma_i^k\lambda_i^k\right)
			    =\ \partial_{\x_i}\ell_i\left(\widehat{\x}_i^k,\widehat{\y}_i^k\right)
			    +\frac{1}{\alpha_i}\left(\widehat{\x}_i^k-\x_i^k\right).
		\end{equation*}
	  On the other hand, the maximum 
	  \begin{align*}
	  	\max_{\y_i\in{\mR^{p+q}}}\left\{\ell_i(\widehat{\x}_i^k,\y_i)-
	  	   \frac{1}{2\gamma_i}\Vert\y_i-(\y_i^k-\gamma_i\lambda_i^k)\Vert^2\right\}
	  \end{align*}
       is found at $\widehat{\y}_i^k$, hence it holds that
       \begin{equation*}
       	0\in\partial_{\y_i}\left(\ell_i(\widehat{\x}_i^k,\y_i)-
       	\frac{1}{2\gamma_i}\Vert\y_i-(\y_i^k-\gamma_i\lambda_i^k)\Vert^2\right)\bigg|
       	  _{\y_i=\widehat{\y}_i^k}=\partial_{\y_i}\ell_i\left(\widehat{\x}_i^k,
       	  \widehat{\y}_i^k\right)-\frac{1}{\gamma_i}\left(\widehat{\y}_i^k-
       	  (\y_i^k-\gamma_i^k\lambda_i^k)\right).
       \end{equation*}
      Consequently, one has
		\begin{align*}
			\Upsilon^{-1}\left(\x^k-\widehat{\x}^k\right)+v^k
			&\in\partial_\x\mL(\widehat{\x}^k,\widehat{\Y}^k),\\
			\Gamma^{-1}\left(\left(\Y^k-\Gamma\Lambda^k\right)-\widehat{\Y}^k\right)
			&\in-\partial_\Y\mL(\widehat{\x}^k,\widehat{\Y}^k),\\
		\end{align*}
	    where $v^k=\mathrm{col}(v_i^k)_{i=1}^m$.
		Combining them with $\widehat{\Z}^k=\Z^k+\beta\U\widehat{\Y}^k$, noticing
		that $\Lambda^k=\U^\top\Z^k$ and $\widehat{\Lambda}^k=\U^\top\widehat{\Z}^k$,
		it holds that
		\begin{equation*}
			\widehat{\xi}^k=(Q+\Phi)^{-1}\left(Q\xi^k+V^k\right),
		\end{equation*}
	  $\xi^{k+1}$ is generated by $\xi^k-M\left(\xi^k-\widehat{\xi}^k\right)$ all the time.
	  The bound $\Vert v_i^k\Vert\le\varepsilon_i^k, \forall\ i\in\V$,
	  indicates that
		\begin{equation*}
			\Vert V^k\Vert=\left(\sum_{i=1}^m\Vert v_i^k\Vert^2\right)^{1/2}
			\le\left(\sum_{i=1}^m(\varepsilon_i^k)^2\right)^{1/2}\triangleq\varepsilon^k.
			\qedhere
		\end{equation*}
	\end{proof}
	\subsection{Sublinear Rate Under General Convexity}
	Lemma {\bf\ref{LEM_2}} reveals that the inexact-DPMM algorithm is essentially a prediction-correction proximal point method with noises $V^k$. Using monotone operator theory, the convergence analysis is templated.
	\begin{thm}\label{THM_1}
		Suppose Assumptions \textup{\bf\ref{ass_1}, \ref{ass_2}} and 
		\textup{\bf\ref{ass_3}} hold.
		If parameters $\theta_i\in(0,2)$, $\alpha_i>0$, $\gamma_i>0$, 
		$0<\gamma_i\beta<\frac{1}{\lambda_{\max}(\text{\rm\L})}$, and 
		$\sum_{k=0}^\infty\varepsilon_i^k<\infty,\forall \ i=1,2,\dots,m$, the sequences
		$\left\{\xi^k\right\}_{k\ge0}$ and $\left\{\widehat{\xi}^k\right\}_{k\ge0}$
		generated by Algorithm {\rm\bf\ref{alg_1}} satisfy
		\begin{enumerate}[{\bf(a)}]
			\item for any $\xi^*=\mathrm{col}\left(\x^*,\Y^*,\Z^*\right)\in\Phi^{-1}(0)$,
			it holds
			\begin{equation}\label{ineq_conv}
				\left\Vert\xi^{k+1}-\xi^*\right\Vert^2_\mH\le
				\left\Vert\xi^k-\xi^*\right\Vert^2_\mH-\left\Vert
				\xi^k-\widehat{\xi}^k\right\Vert^2_\D
				+2\left\langle\widehat{\xi}^k-\xi^*,V^k\right\rangle,
			\end{equation}
			where $V^k$ is defined by Lemma {\bf\ref{LEM_2}}.
			\item the sequence $\left\{\xi^k\right\}_{k\ge0}$ is bounded.
			\item the sequence $\left\{\xi^k\right\}_{k\ge0}$ converges to a zero-point $\xi^\infty=\mathrm{col}\left(\x^\infty,\Y^\infty,\Z^\infty\right)
			\in\Phi^{-1}(0)$, i.e., $\x^\infty$ is an optimal solution to \eqref{P},
			$\Y^\infty=\1\otimes\y^\infty$, and $\y^\infty$ is an optimal solution to \eqref{D}.
		\end{enumerate}
	\end{thm}
	\begin{proof}
		\hspace{1em}
		\begin{enumerate}[{\bf(a)}]
			\item From \eqref{error_PPA} in Lemma {\bf\ref{LEM_2}}, we have 
			$Q\left(\xi^k-\widehat{\xi}^k\right)+V^k
			\in\Phi\left(\widehat{\xi}^k\right)$ and
			$Q\left(\xi^k-\widehat{\xi}^k\right)
			=\mH\left(\xi^k-\xi^{k+1}\right)$. For any $\xi^*\in\Phi^{-1}(0)$,
			Using the monotonicity of the operator $\Phi$, it holds that
			\begin{equation}\label{monot}
				\left\langle\mH\left(\xi^k-\xi^{k+1}\right)+V^k,
				\widehat{\xi}^k-\xi^*\right\rangle\ge0.
			\end{equation}
			Applying the identity $2\langle a-b,H(c-d)\rangle=\Vert a-d\Vert^2_\mH-\Vert a-c\Vert^2_\mH+\Vert b-c\Vert^2_\mH-\Vert b-d\Vert^2_\mH$ to 
			$\left\langle\mH\left(\xi^k-\xi^{k+1}\right),\widehat{\xi}^k-\xi^*\right\rangle$,
			we have 
			\begin{equation}\label{identy}
				\begin{aligned}
					\left\langle\mH\left(\xi^k-\xi^{k+1}\right),\widehat{\xi}^k-\xi^*\right\rangle
					&=\frac{1}{2}\left\Vert\widehat{\xi}^k-\xi^{k+1}\right\Vert^2_\mH-\frac{1}{2}
					\left\Vert\widehat{\xi}^k-\xi^k\right\Vert^2_\mH\\
					&\ +\frac{1}{2}\left\Vert\xi^*-\xi^k\right\Vert^2_\mH
					-\frac{1}{2}\left\Vert\xi^*-\xi^{k+1}\right\Vert^2_\mH.
				\end{aligned}
			\end{equation}
			Since $Q=\mH M,\ \D=Q^\top+Q-M^\top\mH M$, it follows that
			\begin{align}
				&\left\Vert\widehat{\xi}^k-\xi^k\right\Vert^2_\mH-
				\left\Vert\widehat{\xi}^k-\xi^{k+1}\right\Vert^2_\mH\notag\\
				=&\left\Vert\widehat{\xi}^k-\xi^k\right\Vert^2_\mH-
				\left\Vert\left(\xi^k-\widehat{\xi}^k\right)-
				M\left(\xi^k-\widehat{\xi}^k\right)\right\Vert^2_\mH\notag\\
				=& 2\left\langle\xi^k-\widehat{\xi}^k,\mH M
				\left(\xi^k-\widehat{\xi}^k\right)\right\rangle
				-\left\langle\xi^k-\widehat{\xi}^k,
				M^\top\mH M\left(\xi^k-\widehat{\xi}^k\right)\right\rangle\notag\\
				=&\left\langle\xi^k-\widehat{\xi}^k,
				\left(Q^\top+Q-M^\top\mH M\right)\left(\xi^k-\widehat{\xi}^k\right)\right\rangle\notag\\
				=&\left\Vert\xi^k-\widehat{\xi}^k\right\Vert^2_\D.\label{norm_D}
			\end{align}
			It is obvious that \eqref{monot}, \eqref{identy}, and \eqref{norm_D} jointly 
			imply \eqref{ineq_conv}.
			
			\item Rewrite \eqref{ineq_conv} as
			\begin{align*}
				\left\Vert\xi^{k+1}-\xi^*\right\Vert^2_\mH&\le
				\left\Vert\xi^k-\xi^*\right\Vert^2_\mH
				+2\left\langle\xi^k-\xi^*,V^k\right\rangle
				+2\left\langle\widehat{\xi}^k-\xi^k,V^k\right\rangle
				-\left\Vert\xi^k-\widehat{\xi}^k\right\Vert^2_\D\\
				&= \left\Vert\xi^k-\xi^*\right\Vert^2_\mH
				+2\left\langle\mH^{1/2}\left(\xi^k-\xi^*\right),
				\mH^{-1/2}V^k\right\rangle\\
				&\hspace{6.3em}+2\left\langle\D^{1/2}\left(\widehat{\xi}^k-\xi^k\right),
				\D^{-1/2}V^k\right\rangle-\left\Vert\xi^k-\widehat{\xi}^k\right\Vert^2_\D.
			\end{align*}
			In view of the basic inequality: 
			$2\langle a,b\rangle\le\Vert a\Vert^2+\Vert b\Vert^2$, hence
			\begin{equation*}
				2\left\langle\D^{1/2}\left(\widehat{\xi}^k-\xi^k\right),
				\D^{-1/2}V^k\right\rangle-\left\Vert\xi^k-\widehat{\xi}^k\right\Vert^2_\D
				\le\left\Vert\D^{-1/2}V^k\right\Vert^2.
			\end{equation*}
			Similarly, the inequality $\langle a,b \rangle\le\Vert a\Vert\Vert b\Vert$ entails
			that 
			\begin{equation*}
				\left\Vert\xi^k-\xi^*\right\Vert^2_\mH
				+2\left\langle\mH^{1/2}\left(\xi^k-\xi^*\right),\mH^{-1/2}V^k\right\rangle
				\le\left(\left\Vert\xi^k-\xi^*\right\Vert_\mH+
				\left\Vert\mH^{-1/2}V^k\right\Vert\right)^2.
			\end{equation*}
			Adding up these bounds yields 
			\begin{align*}
				\left\Vert\xi^{k+1}-\xi^*\right\Vert^2_\mH&\le
				\left(\left\Vert\xi^k-\xi^*\right\Vert_\mH+
				\left\Vert\mH^{-1/2}V^k\right\Vert\right)^2
				+\left\Vert\D^{-1/2}V^k\right\Vert^2\\
				&\le\bigg(\left\Vert\xi^k-\xi^*\right\Vert_\mH+
				\left\Vert\mH^{-1/2}V^k\right\Vert+\left\Vert\D^{-1/2}V^k\right\Vert\bigg)^2.
			\end{align*}
			Noticing that $\Vert V^k\Vert\le\varepsilon^k$, furthermore, it holds that
			\begin{equation}\label{xi_bounded}
				\left\Vert\xi^{k+1}-\xi^*\right\Vert_\mH\le
				\left\Vert\xi^k-\xi^*\right\Vert_\mH+\mu\varepsilon^k,
			\end{equation}
			where $\mu=\Vert\mH^{-1/2}\Vert+\Vert\D^{-1/2}\Vert$. Since the non-negative scalar sequence $\left\{\varepsilon^k\right\}_{k\ge0}$ is summable, for any positive integer
			$l$, Summing $k$ from $0$ to $l$ at both sides of \eqref{xi_bounded} yields
			\begin{equation*}
				\left\Vert\xi^{l+1}-\xi^*\right\Vert_\mH\le\left(
				\left\Vert\xi^0-\xi^*\right\Vert_\mH+\mu\sum_{k=0}^\infty\varepsilon^k\right)
				<\infty,\quad\forall \ l\ge0.
			\end{equation*}
			It implies the boundedness of the sequence $\left\{\xi^k\right\}_{k\ge0}$.
			
			\item The bound \eqref{xi_bounded} means that 
			$\left\{\Vert\xi^k-\xi^*\Vert_\mH\right\}_{k\ge0}$ 
			is a quasi-Fej\'er monotone sequence, so it is convergent.
			Since $\{\xi^k\}_{k\ge0}$ is bounded, let $\xi^\infty$ be an arbitrary cluster.
			The relation $Q\left(\xi^k-\widehat{\xi}^k\right)
			=\mH\left(\xi^k-\xi^{k+1}\right)\to 0$ indicates that $\xi^\infty$ is also a cluster
			of the sequence $\left\{\widehat{\xi}^k\right\}_{k\ge0}$. Taking the limit in the
			inclusion $Q\left(\xi^k-\widehat{\xi}^k\right)+V^k\in\Phi\left(\widehat{\xi}^k\right)$,
			we have $0\in\Phi\left(\xi^\infty\right)$. It means that $\xi^\infty$ and
			$\xi^*$ plays the same role in our deductions. In parts {\bf(a)} and {\bf(b)}, if 
			we replace $\xi^*$ with $\xi^\infty$, the results are still valid. Thus the sequence
			$\left\{\Vert\xi^k-\xi^\infty\Vert_\mH\right\}_{k\ge0}$ is convergent, and it converges only to $0$, i.e., $\xi^k\to\xi^\infty$ as $k\to\infty$.
			Suppose $\xi^\infty=\mathrm{col}\left(\x^\infty,\Y^\infty,\Z^\infty\right)$.
			From Lemma {\bf\ref{LEM_1}}, we know that there exists $\y^\infty$ such that $\Y^\infty=\1\otimes \y^\infty$, and $\x^\infty$ and $\y^\infty$ are an optimal 
			solution to \eqref{P} and \eqref{D}, respectively.
			\qedhere
		\end{enumerate}
	\end{proof} 
	The convergence results of Theorem {\bf\ref{THM_1} (c)} show that for each agent
	$i\in\{1,2,\dots,m\}$, its decision variable sequence $\left\{\x_i^k\right\}_{k\ge0}$ 
	converges to $\x_i^\infty$,
	and Lagrange multiplier estimate sequence $\left\{\y_i^k\right\}_{k\ge0}$ 
	converges to  $\y^\infty$. In other words, by communicating with their neighbors and 
	computing locally, all agents cooperatively find a primal optimal solution 
	$\x^\infty$ to \eqref{P}, and reach a consensus on the estimates of the Lagrange 
	multiplier. This consensual multiplier $\y^\infty$ is also a dual optimal solution to 
	\eqref{D}. However, some distributed optimization algorithms for \eqref{P} in the literature have either no knowledge of the convergence of its generated sequence 
	$\left\{\x^k\right\}_{k\ge0}$ or the weaker result that each cluster of 
	$\left\{\x^k\right\}_{k\ge0}$ is an optimal solution to \eqref{P} 
	(cf. Table {\bf\ref{table_1}}).
	
	The first-order optimality residual (KKT system) is one of the most important quantities to measure the optimality of the solution generated by convex optimization algorithms. The following theorem presents the convergence rate of the first-order optimality residual of the proposed Algorithm {\bf\ref{alg_1}}.
	\begin{thm}\label{THM_2}
		Under the same Assumptions and parameters as in Theorem {\bf\ref{THM_1}},
		it holds that
		\begin{equation}\label{ineq_rate}
			\begin{aligned}
				&\left\Vert M\left(\xi^{k+1}-\widehat{\xi}^{k+1}\right)\right\Vert^2_\mH+
				\left\Vert \left(\xi^k-\widehat{\xi}^k\right)
				-\left(\xi^{k+1}-\widehat{\xi}^{k+1}\right)\right\Vert^2_\D \\
				\le \ &\left\Vert M\left(\xi^k-\widehat{\xi}^k\right)\right\Vert^2_\mH
				+2\left\langle\widehat{\xi}^k-\widehat{\xi}^{k+1},V^k-V^{k+1}
				\right\rangle.
			\end{aligned}
		\end{equation}
		Furthermore, if $\varepsilon_i^k=O((\tau_i)^k)$, 
		scalars $\tau_i\in(0,1),\forall\ i=1,2,\dots,m$.
		Let $\tau_{\max}=\mathop{\max}\limits_i\tau_i\in(0,1)$, we then have
		\begin{enumerate}[{\bf(a)}]
			\item successive difference:
			\begin{equation*}
				\left\Vert\xi^{k+1}-\xi^k\right\Vert^2_\mH=o\left(\frac{1}{k}\right)
				+O\left((\tau_{\max})^k\right).
			\end{equation*}
			\item first-order optimality residual:
			\begin{equation*}
				\mathrm{dist}^2\left(0,\Phi\left(\widehat{\x}^k,\widehat{\Y}^k,
				\widehat{\Z}^k\right)\right)=o\left(\frac{1}{k}\right)
				+O\left((\tau_{\max})^k\right).
			\end{equation*}
		\end{enumerate}
	\end{thm}
	\begin{proof}
		In the analysis of convergence rate,
		we will use a common result \cite[Proposition 1]{PG-EXTRA}:
		Let $\left\{r_k\right\}_{k\ge0}$ be a non-negative scalar sequence, one has
		\begin{equation}\label{sublinear}
			\begin{aligned}
				\left.
				\begin{array}{c}
					0<r_{k+1}\le r_k,\ \forall \ k\ge0\\
					\sum_{k=0}^\infty r_k<\infty
				\end{array}
				\right\}
				\Longrightarrow
				r_k=o\left(\frac{1}{k}\right).
			\end{aligned}
		\end{equation}

		From Theorem {\bf\ref{THM_1}}, we know that the sequences $\left\{\xi^k\right\}_{k\ge0}$ and $\left\{\widehat{\xi}^k\right\}_{k\ge0}$
		are bounded. Suppose that $\left\Vert\widehat{\xi}^k-\xi^*\right\Vert\le\mu_1$
		for some constant $\mu_1>0$. Summing over $k$ from $0$ to $\infty$
		in \eqref{ineq_conv}, we have
		\begin{equation}\label{bounded_xi_hat}
			\sum_{k=0}^\infty\left\Vert\widehat{\xi^k}-\xi^k\right\Vert^2_\D
			\le\left\Vert\xi^0-\xi^*\right\Vert^2_\mH+
			2\mu_1\sum_{k=0}^\infty\varepsilon^k<\infty.
		\end{equation}
		Noticing that $\xi^k-\xi^{k+1}=M\left(\xi^k-\widehat{\xi}^k\right)$, and 
		the matrices $\mH$ and $\D$ are positive definite, we conclude that 
		the sequence $\left\{\Vert M(\xi^k-
		\widehat{\xi}^k)\Vert_{\mH}\right\}_{k\ge0}$ is summable. Next, we will construct
		a non-negative scalar sequence $\{r_k\}$ satisfying \eqref{sublinear} to complete the analysis of convergence rate. Using again the relation $Q\left(\xi^k-\widehat{\xi}^k\right)+V^k\in\Phi\left(\widehat{\xi^k}\right)$
		and the monotonicity of the operator $\Phi$, it follows that
		\begin{equation}\label{eq_E2}
			2\left\langle\widehat{\xi}^k-\widehat{\xi}^{k+1},
			Q\left(\left(\xi^k-\widehat{\xi}^k\right)-
			\left(\xi^k-\widehat{\xi}^k\right)\right)\right\rangle
			+2\left\langle\widehat{\xi}^k-\widehat{\xi}^{k+1},V^k-V^{k+1}\right\rangle\ge0.
		\end{equation}
		Since $2\xi^\top Q\xi=\xi^\top(Q^\top+Q)\xi$, $Q=\mH M,\ 
		\D=Q^\top+Q-M^\top\mH M$, we obtain
		\begin{align}
			&2\left\langle\widehat{\xi}^k-\widehat{\xi}^{k+1},
			Q\left(\left(\xi^k-\widehat{\xi}^k\right)-
			\left(\xi^{k+1}-\widehat{\xi}^{k+1}\right)\right)\right\rangle\notag\\
			= \ &2\left\langle\xi^k-\xi^{k+1},
			Q\left(\left(\xi^k-\widehat{\xi}^k\right)-
			\left(\xi^{k+1}-\widehat{\xi}^{k+1}\right)\right)\right\rangle\notag\\
			&-
			2\left\langle\xi^k-\widehat{\xi}^k-\left(\xi^{k+1}-\widehat{\xi}^{k+1}\right),
			Q\left(\left(\xi^k-\widehat{\xi}^k\right)-
			\left(\xi^{k+1}-\widehat{\xi}^{k+1}\right)\right)\right\rangle\notag\\
			= \ &2\left\langle M\left(\xi^k-\widehat{\xi}^k\right),
			HM\left(\left(\xi^k-\widehat{\xi}^k\right)-
			\left(\xi^k-\widehat{\xi}^k\right)\right)\right\rangle
			-\left\Vert\left(\xi^k-\widehat{\xi}^k\right)-
			\left(\xi^{k+1}-\widehat{\xi}^{k+1}\right)\right\Vert^2_{Q^\top+Q}\notag\\
			= \ &2\left\langle \xi^k-\widehat{\xi}^k,
			M^\top HM\left(\left(\xi^k-\widehat{\xi}^k\right)-
			\left(\xi^k-\widehat{\xi}^k\right)\right)\right\rangle-	
			\left\Vert\left(\xi^k-\widehat{\xi}^k\right)-
			\left(\xi^{k+1}-\widehat{\xi}^{k+1}\right)\right\Vert^2_{Q^\top+Q}.\label{eq_E3}
		\end{align}
		Let $a=M\left(\xi^k-\widehat{\xi}^k\right)$ and
		$b=M\left(\xi^{k+1}-\widehat{\xi}^{k+1}\right)$, it hods that
		
		\begin{align}
			&2\langle b,\mH(a-b)\rangle\notag\\
			= \ &2\left\langle M\left(\xi^{k+1}-\widehat{\xi}^{k+1}\right),
			HM\left(\left(\xi^k-\widehat{\xi}^k\right)-
			\left(\xi^{k+1}-\widehat{\xi}^{k+1}\right)\right)\right\rangle\notag\\
			= \ &-2\left\langle M\left(\left(\xi^k-\widehat{\xi}^k\right)-
			\left(\xi^{k+1}-\widehat{\xi}^{k+1}\right)\right),
			HM\left(\left(\xi^k-\widehat{\xi}^k\right)-
			\left(\xi^{k+1}-\widehat{\xi}^{k+1}\right)\right)\right\rangle\notag\\
			&+2\left\langle M\left(\xi^k-\widehat{\xi}^k\right),
			HM\left(\left(\xi^k-\widehat{\xi}^k\right)-
			\left(\xi^k-\widehat{\xi}^k\right)\right)\right\rangle\notag\\
			\overset{\eqref{eq_E3}}{=}& -2\left\Vert \left(\xi^k-\widehat{\xi}^k\right)
			-\left(\xi^{k+1}-\widehat{\xi}^{k+1}\right)\right\Vert^2_{M^\top\mH M}
			+\left\Vert \left(\xi^k-\widehat{\xi}^k\right)
			-\left(\xi^{k+1}-\widehat{\xi}^{k+1}\right)\right\Vert^2_{Q^\top+Q}\notag\\
			&+2\left\langle\widehat{\xi}^k-\widehat{\xi}^{k+1},
			Q\left(\left(\xi^k-\widehat{\xi}^k\right)-
			\left(\xi^{k+1}-\widehat{\xi}^{k+1}\right)\right)\right\rangle\notag\\
			\overset{\eqref{eq_E2}}{\ge}& \left\Vert \left(\xi^k-\widehat{\xi}^k\right)
			-\left(\xi^{k+1}-\widehat{\xi}^{k+1}\right)\right\Vert^2_{Q^\top+Q-2M^\top\mH M}
			-2\left\langle\widehat{\xi}^k-\widehat{\xi}^{k+1},V^k-V^{k+1}\right\rangle.
			\label{eq_E4}
		\end{align}
		Applying the identity $2\langle b,\mH(a-b)\rangle=\Vert a\Vert^2_\mH-\Vert 
		b\Vert^2_\mH-\Vert a-b\Vert^2_\mH$ yields
		\begin{align*}
			&\hspace{0.2em}\Vert a\Vert^2_\mH-\Vert 
			b\Vert^2_\mH-\Vert a-b\Vert^2_\mH\\
			=&\left\Vert M\left(\xi^k-\widehat{\xi}^k\right)\right\Vert^2_\mH-
			\left\Vert M\left(\xi^{k+1}-\widehat{\xi}^{k+1}\right)\right\Vert^2_\mH-
			\left\Vert M\left(\xi^k-\widehat{\xi}^k\right)-
			M\left(\xi^{k+1}-\widehat{\xi}^{k+1}\right)\right\Vert^2_\mH\notag\\
			=\ & \left\Vert M\left(\xi^k-\widehat{\xi}^k\right)\right\Vert^2_\mH-
			\left\Vert M\left(\xi^{k+1}-\widehat{\xi}^{k+1}\right)\right\Vert^2_\mH-
			\left\Vert \left(\xi^k-\widehat{\xi}^k\right)-
			\left(\xi^{k+1}-\widehat{\xi}^{k+1}\right)\right\Vert^2_{M^\top\mH M}\notag\\
			=\ & 2\langle b,\mH(a-b)\rangle\notag\\
			\overset{\eqref{eq_E4}}{\ge}\hspace{-0.2em}& \ 
			\left\Vert \left(\xi^k-\widehat{\xi}^k\right)
			-\left(\xi^{k+1}-\widehat{\xi}^{k+1}\right)\right\Vert^2_{Q^\top+Q-2M^\top\mH M}
			-2\left\langle\widehat{\xi}^k-\widehat{\xi}^{k+1},V^k-V^{k+1}\right\rangle.
		\end{align*}
		The last inequality can be rewritten as
		\begin{align*}
			&\left\Vert M\left(\xi^{k+1}-\widehat{\xi}^{k+1}\right)\right\Vert^2_\mH+
			\left\Vert \left(\xi^k-\widehat{\xi}^k\right)
			-\left(\xi^{k+1}-\widehat{\xi}^{k+1}\right)\right\Vert^2_\D \\
			\le \ &\left\Vert M\left(\xi^k-\widehat{\xi}^k\right)\right\Vert^2_\mH
			+2\left\langle\widehat{\xi}^k-\widehat{\xi}^{k+1},V^k-V^{k+1}\right\rangle,
		\end{align*}
		which is exactly \eqref{ineq_rate}. The rest is to prove the $o(1/k)$ 
		convergence rate.
		
		\begin{enumerate}[{\bf(a)}]
			\item \emph{Successive difference}: since the sequence 
			$\left\{\widehat{\xi}^k\right\}_{k\ge0}$ is bounded, $\varepsilon_i^k=O((\tau_i)^k)$,
			and $\Vert V^k\Vert\le\varepsilon^k$, without loss of generality, 
			we assume that there exists constant $\widehat{\mu}>0$ such that
			\begin{equation*}
				\left\vert2\left\langle\widehat{\xi}^k-\widehat{\xi}^{k+1},
				V^k-V^{k+1}\right\rangle\right\vert\le\widehat{\mu}(\tau_{\max})^k.
			\end{equation*}
			Using the error bound \eqref{ineq_rate}, we then further obtain
			\begin{align*}
				\left\Vert M\left(\xi^{k+1}-\widehat{\xi}^{k+1}\right)\right\Vert^2_\mH
				\le\left\Vert M\left(\xi^k-\widehat{\xi}^k\right)\right\Vert^2_\mH
				+\widehat{\mu}(\tau_{\max})^k,
			\end{align*}
			which is equivalent to
			\begin{align*}
				\left\Vert M\left(\xi^{k+1}-\widehat{\xi}^{k+1}\right)\right\Vert^2_\mH
				+\frac{\widehat{\mu}}{1-\tau_{\max}}(\tau_{\max})^{k+1}
				\le\left\Vert M\left(\xi^k-\widehat{\xi}^k\right)\right\Vert^2_\mH
				+\frac{\widehat{\mu}}{1-\tau_{\max}}(\tau_{\max})^k.
			\end{align*}
			Denote 
			\begin{equation*}
				r_k=\left\Vert M\left(\xi^k-\widehat{\xi}^k\right)\right\Vert^2_\mH
				+\frac{\widehat{\mu}}{1-\tau_{\max}}(\tau_{\max})^k.
			\end{equation*}
			Now we have $0<r_{k+1}\le r_k, \forall\ k\ge0$. In addition, the facts that 
			$\tau_{\max}\in(0,1)$ and  the sequence $\left\{\Vert M(\xi^k-
			\widehat{\xi}^k)\Vert_{\mH}\right\}_{k\ge0}$ is summable jointly indicate that
			$\sum_{k=0}^\infty r_k<\infty$.
			According to \eqref{sublinear}, we have
			\begin{equation*}
				r_k=\left\Vert M\left(\xi^k-\widehat{\xi}^k\right)\right\Vert^2_\mH
				+\frac{\widehat{\mu}}{1-\tau_{\max}}(\tau_{\max})^k=o\left(\frac{1}{k}\right),
			\end{equation*}
			which implies that
			\begin{equation*}
				\left\Vert\xi^k-\xi^{k+1}\right\Vert^2_\mH=\left\Vert M\left(\xi^k-\widehat{\xi}^k\right)\right\Vert^2_\mH
				=o\left(\frac{1}{k}\right)+O\left((\tau_{\max})^k\right).
			\end{equation*}
			
			\item \emph{First-order optimality residual}: note that
			$\mH\left(\xi^k-\xi^{k+1}\right)+V^k\in
			\Phi\left(\widehat{\x}^k,\widehat{\Y}^k,
			\widehat{\Z}^k\right)$, and 
			\begin{equation*}
				\left\Vert M\left(\xi^k-\widehat{\xi}^k\right)\right\Vert^2_\mH
				+\frac{\widehat{\mu}}{1-\tau_{\max}}(\tau_{\max})^k=o\left(\frac{1}{k}\right),
			\end{equation*}
			hence it holds that
			\begin{align*}
				\mathrm{dist}^2\left(0,\Phi\left(\widehat{\x}^k,\widehat{\Y}^k,
				\widehat{\Z}^k\right)\right)
				\le \ &\left\Vert \mH\left(\xi^k-\xi^{k+1}\right)
				+V^k\right\Vert^2\\
				\le \ & \left\Vert\xi^k-\xi^{k+1}\right\Vert^2_\mH
				+2\left\Vert\mH\left(\xi^k-\xi^{k+1}\right)\right\Vert\left\Vert V^k\right\Vert
				+\left\Vert V^k\right\Vert^2\\
				= \ &o\left(\frac{1}{k}\right)+O((\tau_{\max})^k).\qedhere
			\end{align*}.
		\end{enumerate}
	\end{proof}
	Theorem {\bf\ref{THM_1}} and {\bf\ref{THM_2}} demonstrate the practicality of the inexact-DPMM algorithm in comparison to those algorithms that assume exact solutions for subproblems. The inexact-DPMM algorithm not only provides a convergence guarantee but also a convergence rate of $o(1/k)$ for the first-order optimality residual. It's obvious that if the subproblem is solved exactly, 
	i.e., $\varepsilon_i^k\equiv0 \text{\ for any} \ i\in\V, k\ge0$, 
	the convergence and rate of the DPMM algorithm remain valid.
	
	\subsection{Linear Rate Under Structural Assumption}
	In this subsection, by some techniques of variational analysis, we will present the linear convergence rate of the proposed DPMM algorithm. The notion of metric subregularity is fundamental for the convergence analysis of proximal algorithms.
	\begin{defi}[metric subregularity, \cite{Rockafellar09}]
		A set-valued mapping $S:\mR^n\rightrightarrows\mR^n$ is said to be metrically 
		subregular at $(\bar{u},\bar{v})\in{\rm gph}(S)$, if for some $\epsilon>0$, there
		exists $\kappa\ge0$ such that
		\begin{equation*}
			\mathrm{dist}(u,S^{-1}(\bar{v}))\le\kappa \ \mathrm{dist}(\bar{v},S(u)), \ 
			\forall \ u\in \mathbf{B}_\epsilon(\bar{u}),
		\end{equation*}
		where $\dist(u,S^{-1}(\bar{v}))\triangleq\inf
		\left\{\Vert u-u^\prime\Vert:\ u^\prime\in S^{-1}(\bar{v})\right\}$, 
		$\mathrm{gph}(S)=\{(u,v):v\in S(u)\}$, $S^{-1}(v)=\{u: v\in S(u)\}$, and
		$\mathbf{B}_\epsilon(\bar{u})=\{u: \Vert u-\bar{u}\Vert\le\epsilon\}$.
	\end{defi}
	If $S^{-1}(\bar{v})$ characterizes the optimal solution set of a given optimization problem, then $\bar{u}\in S^{-1}(\bar{v})$ is an optimal solution. The metric subregularity serves as an error bound condition at the optimal solution $\bar{u}$, it plays an important role in the convergence rate analysis of optimization algorithms. To ensure that this metric subregularity holds in this paper, we need the following structural assumption, which is also considered in References \cite{YeJJ21, Yuan20, Luo92}, etc.
	\begin{ass}[Structural assumption for \eqref{P}] \label{ass_4}
		\hspace{1em}
		\begin{enumerate}[{\bf(a)}]
			\item The local objective $f_i(\x_i)=h_i(B_i\x_i)+\langle q_i,\x_i\rangle+r_i(\x_i)$,
			where $h_i$ is a smooth and essentially locally strongly convex function,
			i.e., $h_i$ is strongly convex on any compact and convex subset.
			$B_i$ and $q_i$ are a given matrix and vector, respectively, and the
			subdifferential $\partial r_i(\x_i)$ of the 
			non-smooth term $r_i$ is a polyhedral set, $\forall \ i=1,2,\dots,m$.
			
			\item The constraint function $G_i(\x_i)=C_i\x_i-d_i$, 
			which holds if each $g_i$
			is affine, $\forall \ i=1,2,\dots,m$.
			\item The local constraint set $\Omega_i$ is a convex polyhedron, $\forall \ i=1,2,\dots,m$.
		\end{enumerate}
		\begin{rem}\label{rem_3}
			The structural assumption {\bf\ref{ass_4}} is not too restrictive, there are many commonly encountered optimization problems meeting it, such as regression problems in machine learning, likelihood estimation in statistics, and constrained LASSO problems. If each $r_i$ is a polyhedral
			convex function, which means that its epigraph 
			$\mathrm{epi}(r_i)=\{(\x_i,\alpha): \alpha\ge r_i(\x_i)\}$ is a  polyhedral set, or $r_i$ is a convex piecewise linear-quadratic function, the
			subdifferential $\partial r_i(\x_i)$ is satisfied to be a polyhedral multifunction. In Table {\bf\ref{table_2}}, we give some specific examples of $h_i$ and $r_i$ that satisfy Assumption {\bf\ref{ass_4}}, respectively.
		\end{rem}
	\end{ass}  
	\begin{table}[htbp]
		\begin{threeparttable}
			\centering
			\caption{Some commonly used functions $h_i$ and $r_i$.}
			\label{table_2}
			\begin{small}
				\begin{tabular}{c|c|c|c|c}
					\toprule
					\tabincell{l}{Obejective
						\\ function} & \tabincell{l}{Linear
						\\ regression}
					& \tabincell{l}{Logistic
						\\ regression} & 
					\tabincell{l}{Likelihood
						\\ estimation} &
					\tabincell{l}{ Poisson
						\\ regression}\\
					\midrule
					$h_i(\x_i)$ & $\frac{1}{2}\Vert\x_i-b\Vert_2^2$ 
					&\tabincell{l}{$\mathop{\sum}\limits_j\log(1+e^{\x_i(j)})$
						\\ $-b^\top\x_i$}
					&$-\mathop{\sum}\limits_j\log(\x_i(j))+b^\top\x_i$
					& $\mathop{\sum}\limits_je^{\x_i(j)}-b^\top\x_i$\\
					\bottomrule
					\specialrule{0em}{3pt}{3pt}
					\toprule
					regularizer & $\ell_1$-norm & elastic net \cite{Zou05} 
					& fused LASSO \cite{Tibshirani04}& OSCAR \cite{Bondell08} \\
					\midrule
					$r_i(\x_i)$ & $\mu\Vert\x_i\Vert_1$ 
					& $\mu_1\Vert\x_i\Vert_1+\mu_2\Vert\x_i\Vert_2^2$ &
					\tabincell{l}{$\mu_1\mathop{\sum}\limits_j\vert\x_i(j)-\x_i(j+1)\vert$ 
						\\ $+\mu_2\Vert\x_i\Vert_1$}
					& \tabincell{l}
					{$\mu_1\mathop{\sum}\limits_{r<l}\max\{\vert\x_i(r)\vert,\vert\x_i(l)\vert\}$
						\\ $+\mu_2\Vert\x_i\Vert_1$}
					\\
					\bottomrule
				\end{tabular}
			\end{small}
		\end{threeparttable}
	\end{table}
	
	Under Assumption {\bf\ref{ass_4}}, the first-order optimality condition \eqref{FOOC}  has the form
	\begin{equation}\label{struc_phi}
		\begin{aligned}
			0\in \Phi(\x,\Y,\Z)=\left(
			\begin{array}{c}
				\widetilde{B}^\top\nabla h(\widetilde{B}\x)+q+\partial r(\x)+
				\widetilde{C}^\top\Y+\mathrm{N}_\Omega(\x)\\
				-\widetilde{C}\x+d+\U^\top\Z+\mathrm{N}_{\K_m^\circ}(\Y)\\
				-\U\Y
			\end{array}\right),
		\end{aligned}
	\end{equation}
	where $\widetilde{B}=\mathrm{diag}(B_i)_{i=1}^m$,  $\widetilde{C}=\mathrm{diag}(C_i)_{i=1}^m$, $h(\x)=\sum_{i=1}^mh_i(\x_i)$,
	$r(\x)=\sum_{i=1}^mr_i(\x_i)$, $q=\mathrm{col}(q_i)_{i=1}^m$, 
	$d=\mathrm{col}(d_i)_{i=1}^m$, $\Omega=\Omega_1\times\cdots\times\Omega_m$, 
	$\K_m^\circ$= $\underbrace{\K^\circ\times\cdots\times\K^\circ}_{m}$. The symbols $\mathrm{N}_\Omega(x)$ and $\mathrm{N}_{\K_m^\circ}(\Y)$ represent the 
	normal cones of the sets $\Omega$ at $\x$ and $\K_m^\circ$ at $\Y$, respectively.
	
	Under the strong convexity assumption of $h$, a fact is that the linear mapping 
	$\x \to \widetilde{B}\x$ is invariant over the optimal solution set $\mathcal{X}^*$ of \eqref{P} \cite[Lemma 2.1]{Luo92}. We then provide an alternative characterization of the first-order optimality condition \eqref{struc_phi}.
	\begin{Prop}\label{PROP_2}
		Suppose that Assumption {\bf\ref{ass_4}} is satisfied, there exists constant
		vectors $\tilde{t}$ and $\widetilde{\zeta}=\widetilde{B}^\top\nabla 
		h(\tilde{t})+q$, such that 
		\begin{align*}
			\M^*\triangleq\Phi^{-1}(0)=\bigg\{\xi=(\x,\Y,\Z):\
			&0\in\widetilde{\zeta}+\partial r(\x)+\widetilde{C}^\top\Y+\mathrm{N}_\Omega(\x),\\
			&0\in-\widetilde{C}\x+d+\U^\top\Z+\mathrm{N}_{\K_m^\circ}(\Y),\\
			&\widetilde{B}\x-\tilde{t}=0,-\U\Y=0
			\bigg\}.
		\end{align*}
	\end{Prop}
	\begin{proof}
		The proof is rather standard by \cite[Lemma 2.1]{Luo92} and thus is omitted here.
	\end{proof}
	We introduce the following perturbed set-valued mapping
	\begin{align*}
		\Psi(p)=\bigg\{\xi:& \ p_1=\widetilde{B}\x-\tilde{t},p_2=-\U\Y,
		p_3\in\widetilde{\zeta}+\partial 
		r(\x)+\widetilde{C}^\top\Y+\mathrm{N}_\Omega(\x),\\
		& p_4\in -\widetilde{C}\x+d+\U^\top\Z+\mathrm{N}_{\K_m^\circ}(\Y)\bigg\},
	\end{align*}
	where $p=(p_1,p_2,p_3,p_4)$. It is obvious that $\Psi(0)=\M^*$.
	Note that the normal cones $\mathrm{N}_\Omega(\x)$ and $\mathrm{N}_{\K_m^\circ}(\Y)$
	are the subdifferential of the indicator function $I_\Omega(x)$ and
	$I_{\K_m^\circ}(\Y)$, respectively, where $\Omega$ and $\K_m^\circ$ are
	polyhedral sets. Combining with the structured assumption, it holds that 
	$\Psi$ is a polyhedral set-valued mapping, as well as its inverse $\Psi^{-1}$.
	According to \cite[Proposition {\bf 1}]{Robinson1981}, $\Psi^{-1}$ is metrically subregular at any point $(p,\xi)\in\mathrm{gph}(\Psi)$.
	On the other hand, similar to 
	\cite[Proposition 40]{Yuan20} and \cite[Proposition 6]{YeJJ21}, we have the following
	equivalence with respect to the metric subregularity between $\Psi^{-1}$ and $\Phi$.
	\begin{Prop}\label{PROP_3}
		Suppose that Assumption {\bf\ref{ass_4}} is satisfied, then the metric subregularity conditions of $\Psi^{-1}$ and $\Phi$ are equivalent. Precisely, 
		given $\xi^*\in\M^*$,  the following two statements are equivalent:
		\begin{enumerate}[{\bf(a)}]
			\item There exists $\kappa_1, \epsilon_1>0$ such that
			\begin{equation*}
				\mathrm{dist}(\xi,\Psi(0))\le\kappa_1\mathrm{dist}(0,\Psi^{-1}(\xi)),
				\forall \ \xi\in\mathbf{B}_{\epsilon_1}(\xi^*).
			\end{equation*}
			\item There exists $\kappa_2, \epsilon_2>0$ such that
			\begin{equation*}
				\mathrm{dist}(\xi,\Phi^{-1}(0))\le\kappa_2\mathrm{dist}(0,\Phi(\xi)),
				\forall \ \xi\in\mathbf{B}_{\epsilon_2}(\xi^*).
			\end{equation*}
		\end{enumerate}
	\end{Prop}
	Since the metric subregularity of $\Psi^{-1}$ holds at any point 
	$(p,\xi)\in\mathrm{gph}(\Psi)$, by
	Proposition {\bf\ref{PROP_3}}, it follows that $\Phi$ is metrically subregular 
	at $(\xi^*,0)$ for any $\xi^*\in\M^*=\Phi^{-1}(0)$. Now, we show the linear 
	convergence rate of DPMM as follows
	\begin{thm}\label{THM_3}
		Under Assumptions {\bf\ref{ass_1}, \ref{ass_2}, \ref{ass_3}}, and {\bf\ref{ass_4}},
		if the algorithmic parameters meet Proposition {\bf\ref{PROP_1}}, and 
		$\varepsilon_i^k\le\delta_i^k\Vert\x_i^{k+1}-\x_i^k\Vert$ with $\sum_{i=1}^m\delta_i^k<+\infty$, $\forall \ i=1,2,\dots,m$, the sequence
		$\{\xi^k\}_{k\ge0}$ generated by Algorithm {\bf\ref{alg_1}} satisfies
		\begin{equation}
			\mathrm{dist}_\mH(\xi^{k+1},\M^*)\le\varrho_k\ \mathrm{dist}_\mH(\xi^k,\M^*),
			\quad\forall \ k\ge\tilde{k},
		\end{equation}
		where $\dist_\mH(\xi^k,\M^*)=\inf_{\xi\in\M^*}\Vert\xi^k-\xi\Vert_\mH$,
		$\tilde{k}$ is a sufficiently large positive integer, 
		\begin{align*}
			\varrho_k &= \frac{\omega c_4\delta^k+\varrho}{1-c_4\delta^k}\to \varrho<1, 
			\quad \text{as} \ k\to\infty,\\ 
			\varrho & = \sqrt{1-\left(\frac{c_3}{c_2(1+\kappa c_1c_2)}\right)^2}<1,
		\end{align*}
		$\kappa$ is the metric subregularity constant of the operator $\Phi$ at $(\xi^*,0)$, 
		$c_1=(\lambda_{\max}(Q^\top Q)/\lambda_{\min}(\D))^{1/2}$,
		$c_2=\sqrt{\lambda_{\max}(\mH)}$, $c_3=\sqrt{\lambda_{\min}(\D)}<c_2$,
		$c_4=\sqrt{\lambda_{\max}(M^\top\mH M)}/\sqrt{\lambda_{\min}(\mH)}$,
		$\omega=1+c_2c_4/c_3$, and $\delta^k=\mathop{\max}\limits_{1\le i\le m}\delta_i^k,
		\ \forall \ k\ge0$.
	\end{thm}
	\begin{proof}
	   	From Lemma {\bf\ref{LEM_2}}, it holds that
	   \begin{align}\label{delta}
	   	\Vert V^k\Vert\le\varepsilon^k\le\delta^k\Vert\x^{k+1}-\x^k\Vert,\quad
	   	\sum_{k=0}^\infty\delta^k<\infty.
	   \end{align}
	   Introduce the following notations
	   \begin{subequations}\label{exact}
	   	\begin{align}
	   		\widetilde{\eta}^k&=(Q+\Phi)^{-1}(Q\xi^k),\label{exact_a}\\
	   		\eta^{k+1}&=\xi^k-M(\xi^k-\widetilde{\eta}^k).\label{exact_b}
	   	\end{align}
	   \end{subequations}
	   They are the exact schemes of DPMM at each iteration $k$. We will estimate the error bound between \eqref{exact} and the inexact DPMM \eqref{error_PPA} to derive the convergence rate. It is obvious that \eqref{exact} is a special case of \eqref{error_PPA} with $V^k=0$. Similar to \eqref{ineq_conv} in Theorem 
	   {\bf\ref{THM_1}}, we have 
	   \begin{equation}\label{eta_bound}
	   	\Vert\eta^{k+1}-\xi^*\Vert^2_\mH\le\Vert\xi^k-\xi^*\Vert^2_\mH
	   	-\Vert\xi^k-\widetilde{\eta}^k\Vert^2_\D, \ \forall \ \xi^*\in\M^*.
	   \end{equation}
	   From Theorem {\bf\ref{THM_1}}, the sequence $\{\xi^k\}_{k\ge0}$ converges to some $\xi^\infty\in\M^*$. In view of \eqref{eta_bound}, it holds that $\eta^{k+1},\ \widetilde{\eta}^k\to\xi^\infty$, 
	   and $\Vert\xi^k-\widetilde{\eta}^k\Vert\to 0$. On the other hand, \eqref{exact_a}
	   implies that $Q(\xi^k-\widetilde{\eta}^k)\in\Phi(\widetilde{\eta}^k)$, thus
	   \begin{align}\label{c_1}
	   	\dist(0,\Phi(\widetilde{\eta}^k))\le\Vert Q\left(\xi^k-
	   	\widetilde{\eta}^k\right)\Vert\le c_1\Vert\xi^k-\widetilde{\eta}^k\Vert_\D, 
	   \end{align}
	   where $c_1=(\lambda_{\max}(Q^\top Q)/\lambda_{\min}(\D))^{1/2}$. since 
	   $\widetilde{\eta}^k\to\xi^\infty$, for any $\epsilon>0$, there exists a positive integer $\tilde{k}>0$ such that $\widetilde{\eta}^k\in\mathbf{B}_\epsilon(\xi^\infty),
	   \forall \ k\ge\widetilde{k}$. Using the metric subregularity of $\Phi$ at 
	   $(\xi^\infty,0)$ (cf. Proposition {\bf\ref{PROP_3}}), there are constants 
	   $\kappa>0, \epsilon>0$ satisfying
	   \begin{align}
	   	&\dist(\widetilde{\eta}^k,\M^*)=\dist(\widetilde{\eta}^k,\Phi^{-1}(0))\notag\\
	   	\le&\ \kappa \ \dist(0,\Phi(\widetilde{\eta}^k))
	   	\overset{\eqref{c_1}}{\le} \kappa c_1 \Vert\xi^k-\widetilde{\eta}^k\Vert_\D, 
	   	\ \forall \ k\ge\widetilde{k}.\label{kappa_c_1}
	   \end{align}
	   Since $\mH\succcurlyeq\D\succ0$, we have 
	   \begin{align}
	   	&\dist(\widetilde{\eta}^k,\M^*)\ge\frac{1}{c_2}\dist_\D(\widetilde{\eta}^k,\M^*)
	   	\notag\\
	   	\ge&\frac{1}{c_2}\bigg(\dist_\D(\xi^k,\M^*)-
	   	\Vert\xi^k-\widetilde{\eta}^k\Vert_\D\bigg),\label{c_2}
	   \end{align}
	   where $c_2=\sqrt{\lambda_{\max}(\mH)}$.
	   Formulas \eqref{kappa_c_1} and \eqref{c_2} jointly imply that
	   \begin{align}
	   	&\dist_\mH(\xi^k,\M^*)\le\frac{c_2}{c_3}\dist_\D(\xi^k,\M^*)\notag\\
	   	\le&\ \frac{c_2(1+\kappa c_1c_2)}{c_3}\Vert\xi^k-\widetilde{\eta}^k\Vert_\D, \
	   	\forall \ k\ge\widetilde{k},\label{c_3}
	   \end{align}
	   where $c_3=\sqrt{\lambda_{\min}(\D)}$. Since $\M^*$ is closed and convex, 
	   then $\dist_\mH(\xi^k,\M^*)=\Vert\xi^k-\xi^*\Vert_{\mH}$ for some $\xi^*\in\M^*$.
	   By \eqref{eta_bound}, it follows that
	   \begin{align*}
	   	&\dist^2_\mH(\xi^k,\M^*)-\dist^2_\mH(\eta^{k+1},\M^*)\ge
	   	\Vert\xi^k-\widetilde{\eta}^k\Vert_\D\\
	   	\ge& \ \left(\frac{c_3}{c_2(1+\kappa c_1c_2)}\right)^2\dist^2_\mH(\xi^k,\M^*), \
	   	\forall \ k\ge\widetilde{k},
	   \end{align*}
	   which entails 
	   \begin{equation}\label{varrho}
	   	\dist_\mH(\eta^{k+1},\M^*)\le\varrho \ \dist_\mH(\xi^k,\M^*), \ \forall \ 
	   	k\ge\widetilde{k},
	   \end{equation}
	   where $\varrho=\sqrt{1-\left(\frac{c_3}{c_2(1+\kappa c_1c_2)}\right)^2}<1$ 
	   (noticing $c_3<c_2$ ). Moreover, From \eqref{eta_bound}, we have
	   \begin{equation*}
	   	\frac{c_3}{c_2}\Vert\xi^k-\widetilde{\eta}^k\Vert_\mH\le
	   	\Vert\xi^k-\widetilde{\eta}^k\Vert_\D\le\dist_\mH(\xi^k,\M^*),\ \forall \ k\ge0,
	   \end{equation*}
	   equivalently,
	   \begin{equation}\label{F9}
	   	\Vert\xi^k-\widetilde{\eta}^k\Vert_\mH\le\frac{c_2}{c_3}\ \dist_\mH(\xi^k,\M^*), 
	   	\forall \ k\ge0.
	   \end{equation}
	   Define the projection $\mP_{\M^*}^\mH(\xi)=\arg\min_{u\in\M^*}\Vert\xi-u\Vert_\mH$.
	   Using the basic inequality $\Vert a+b\Vert\le\Vert a\Vert+\Vert b\Vert$
	   and the non-expansiveness of the projection mapping $\mP_{\M^*}^\mH$, we have
	   \begin{align}
	   	\Vert\xi^k-\mP_{\M^*}^\mH(\eta^{k+1})\Vert_\mH\le&\Vert\mP_{\M^*}^\mH(\xi^k)
	   	-\mP_{\M^*}^\mH(\eta^{k+1})\Vert_\mH+\dist_\mH(\xi^k,\M^*)\notag\\
	   	\le& \Vert\xi^k-\eta^{k+1}\Vert_\mH+\dist_\mH(\xi^k,\M^*)\notag\\
	   	=& \Vert M(\xi^k-\widetilde{\eta}^k)\Vert_\mH+\dist_\mH(\xi^k,\M^*)\notag\\
	   	\overset{\eqref{F9}}{\le}&\omega\ \dist_\mH(\xi^k,\M^*),\label{omega}
	   \end{align}
	   where $\omega=1+c_2c_4/c_3$. A fact is that the variable metric proximal operator $(Q+\Phi)^{-1}$ is non-expansive, thus it holds 
	   \begin{align*}
	   	\Vert\xi^{k+1}-\eta^{k+1}\Vert_\mH&=
	   	\Vert M(Q+\Phi)^{-1}(Q\xi^k+V^k)-M(Q+\Phi)^{-1}(Q\xi^k)\Vert_\mH\\
	   	&\le c_4 \Vert V^k\Vert_\mH\overset{\eqref{delta}}{\le}
	   	c_4\delta^k\Vert\x^{k+1}-\x^k\Vert_\mH
	   	\le c_4\delta^k\Vert\xi^{k+1}-\xi^k\Vert_\mH,
	   \end{align*}
	   where $c_4=\sqrt{\lambda_{\max}(M^\top\mH M)}/\sqrt{\lambda_{\min}(\mH)}$.
	   Using the basic triangle inequality again, we have
	   \begin{align*}
	   	&\Vert\xi^{k+1}-\mP_{\M^*}^\mH(\eta^{k+1})\Vert_\mH
	   	\le\ \Vert\xi^{k+1}-\eta^{k+1}\Vert_\mH+\dist_\mH(\eta^{k+1},\M^*)\\
	   	\le&\ c_4\delta^k\Vert\xi^{k+1}-\xi^k\Vert_\mH+\dist_\mH(\eta^{k+1},\M^*)\\
	   	\le&\ c_4\delta^k\bigg(\Vert\xi^{k+1}-\mP_{\M^*}^\mH(\eta^{k+1},\M^*)\Vert_\mH+
	   	\Vert\xi^k-\mP_{\M^*}^\mH(\eta^{k+1})\Vert_\mH\bigg)+\dist_\mH(\eta^{k+1},\M^*).
	   \end{align*}
	   Noticing that
	   \begin{equation*}
	   	\Vert\xi^{k+1}-\mP_{\M^*}^\mH(\eta^{k+1})\Vert_\mH\ge\dist_\mH(\xi^{k+1},\M^*).
	   \end{equation*}
	   Combining \eqref{omega} with \eqref{varrho}, we obtain
	   \begin{align*}
	   	&\hspace{-0.5em}(1-c_4\delta^k)\dist_\mH(\xi^{k+1},\M^*)\\
	   	\le\ &\ c_4\delta^k\Vert\xi^k-\mP_{\M^*}^\mH(\eta^{k+1})\Vert_\mH
	   	+\dist_\mH(\eta^{k+1},\M^*)\\
	   	\overset{\eqref{omega}}{\le}&\omega c_4\delta^k\dist_\mH(\xi^k,\M^*)+
	   	\dist_\mH(\eta^{k+1},\M^*)\\
	   	\overset{\eqref{varrho}}{\le}
	   	&\left(\omega c_4\delta^k+\varrho\right)\dist_\mH(\xi^k,\M^*), \ 
	   	\forall \ k\ge\widetilde{k},
	   \end{align*}
	   which means that
	   \begin{equation*}
	   	\dist_\mH(\xi^{k+1},\M^*)\le\underbrace{\frac{\omega 
	   			c_4\delta^k+\varrho}{1-c_4\delta^k}}_{\triangleq\varrho_k}
	   	\dist_\mH(\xi^k,\M^*),
	   	\ \forall \ k\ge\widetilde{k}.
	   \end{equation*}
	   The condition $\sum_{k=0}^\infty\delta^k<\infty$ implies $\lim_{k\to\infty}\delta^k=0$.
	   Consequently, it's obvious that 
	   \begin{equation*}
	   	\varrho_k=\frac{c_4\delta^k+\varrho}{1-c_4\delta^k}\to\varrho<1,
	   	\ \text{as}\ k\to\infty, 
	   \end{equation*}
	   which completes the proof.
	\end{proof}
	Theorem {\bf\ref{THM_3}} shows the linear convergence rate of our proposed DPMM 
	algorithm. If the subproblems are solved exactly by agents, i.e., $\delta^k\equiv0$ for all $k\ge0$, the linear convergence rate can be improved as
	\begin{equation*}
		\mathrm{dist}_\mH(\xi^{k+1},\M^*)\le\varrho\ \mathrm{dist}_\mH(\xi^k,\M^*),
		\quad\forall \ k\ge\tilde{k}.
	\end{equation*}
	\begin{rem}
		The condition $\sum_{i=1}^m\delta_i^k<+\infty$ guarantees that the sequence
		$\{\varepsilon_i^k\}_{k\ge0}$ is summable, hence the convergence results of 
		Theorems {\bf\ref{THM_1}} and {\bf\ref{THM_2}} are still valid for Theorem
		{\bf\ref{THM_3}}. It means that $\xi^k$ converges to some $\xi^\infty \in\M^*$.
		The inexact criterion 
		\begin{equation*}
			\dist(0,\partial_{\x_i}\phi_i^k(\widehat{\x}_i^k,\y_i^k-\gamma_i^k\lambda_i^k))
			\le\delta_i^k\Vert\x_i^{k+1}-\x_i^k\Vert \ \text{with} \ \sum_{i=1}^m\delta_i^k<+\infty
		\end{equation*}
		for computing $\widehat{x}_i^k\approx\mathop{\arg\min}\limits_{\x_i\in\mR^{n_i}}\phi_i^k
		(\x_i,\y_i^k-\gamma_i\lambda_i^k)$ is also considered in 
		{\rm \cite{Rockafellar76, Tao18}}, etc,
		but it is computationally unimplementable, because the computation of $\x_i^{k+1}$ happens after $\widehat{\x}_i^k$. From Theorem {\bf\ref{THM_1}},
		$\x^k$ converges to an optimal solution $\x^\infty$ of \eqref{P}, hence
		$\lim_{k\to\infty}\left(\Vert\x_i^{k+1}-
		\x_i^k\Vert-\Vert\x_i^k-\x_i^{k-1}\right)=0$. Therefore, we can use $\Vert\x_i^k-\x_i^{k-1}\Vert$ instead of $\Vert\x_i^{k+1}-\x_i^k\Vert$ in practice.
	\end{rem}
	
	\section{Numerical Simulation}\label{sec_5}
	In this section, we will compare the practical performance of the DPMM algorithm proposed in this paper with some alternative algorithms listed in Table 1 using two numerical examples. Since some of these algorithms cannot manage coupled nonlinear inequality constraints, we split the numerical simulation into two parts.
	\begin{figure}[htbp]
		\centering
		\includegraphics[scale=0.4]{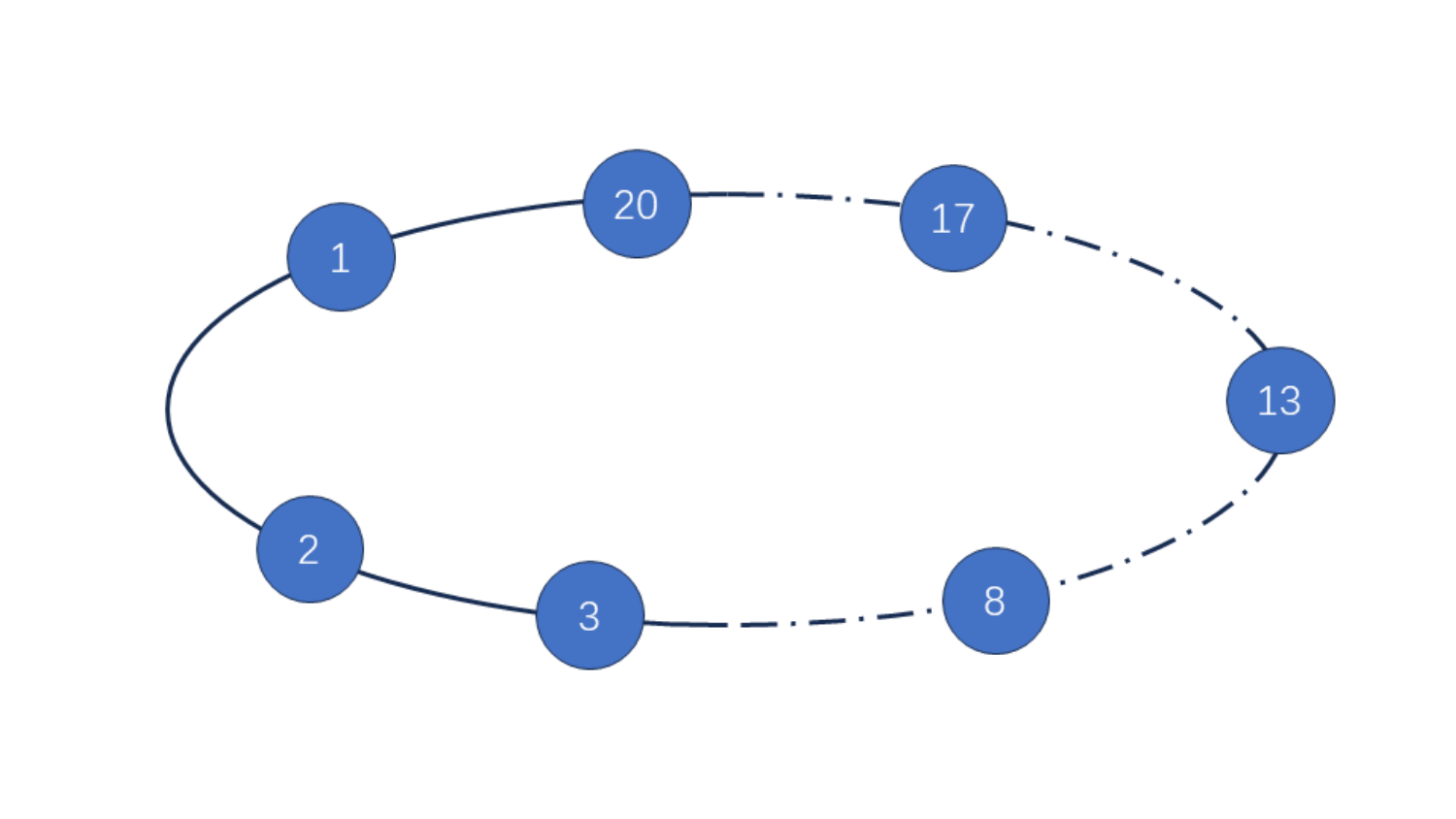}
		\caption{\ Network topology of $20$ agents}
		\label{fig_1}
	\end{figure}
	
	\emph{Network}: The experiments are conducted over a fixed undirected
	connected graph $\G$ which is shown in Fig. {\bf\ref{fig_1}}. The
	adjacency matrix $W$ is constructed as
	\begin{equation*}
		W_{ij}=\left\{
		\begin{aligned}
			\frac{1}{{\rm deg}(i)+1}&,\ \text{if}\ i\neq j, (i,j)\in\E,\\
			0\hspace{2em}&,\ \text{if} \ i\neq j, (i,j)\notin\E,\\
			1-\sum_{j\in\N_i}W_{ij}&, \ \text{if}\ i=j,
		\end{aligned}
		\right.
	\end{equation*}
	where $\mathrm{deg}(i)$ is the degree of agent $i$. For DPMM, we use
	$\text{\L}=(I-W)/2$, and simply set 
	parameters $\theta_i\equiv\theta$, $\alpha_i\equiv\alpha$, 
	$\gamma_i\equiv\gamma$ in their theoretical ranges presented in 
	Proposition {\bf\ref{PROP_1}}.
	
	\emph{Subproblem}: The following two examples are both constrained convex optimization 
	problems with $\ell_1$-norm regularization. If we use DPMM and other  
	alternative algorithms listed in Table {\bf\ref{table_1}} to solve them, the objectives of their subproblems have a unified form: 
	$\psi_i(\x_i)=s_i(\x_i)+\lambda_i\Vert\x_i\Vert_1+\delta_{\Omega_i}(\x_i)$, where
	$\delta_{\Omega_i}(\x_i)$ is the indicator function on a box subset $\Omega_i$, and
	$s_i(\x_i)$ is the smooth part of $\psi_i$. A fact is that both the proximal operators of the functions $\delta_{\Omega_i}(\cdot)$ and $\Vert\cdot\Vert_1$ have closed-form solutions. We then adopt the Davis-Yin splitting algorithm \cite{D-Y17} to minimize $\psi_i(\x_i)$. The subdifferential $\partial\psi_i(\x_i)=\nabla 
	s_i(x_i)+\lambda_i\partial\Vert\x_i\Vert_1+\mathrm{N}_{\Omega_i}(\x_i)$,
	where the sum $\lambda_i\partial\Vert\x_i\Vert_1+\mathrm{N}_{\Omega_i}(\x_i)$ is
	a polyhedral set, precisely, $\lambda_i\partial\Vert\x_i\Vert_1+\mathrm{N}_{\Omega_i}(\x_i)=[L_i,U_i]$, the values 
	of $L_i$ and $U_i$ depend on $\x_i$. At each iteration $k$, we use the following criterion (cf. \eqref{subproblem}) as the stopping condition to calculate $\x_i^{k+1}$ such that $\x_i^{k+1}=\arg\min\psi_i^k(\x_i)=s_i^k(\x_i)+
	\lambda_i\Vert\x_i\Vert_1+\delta_{\Omega_i}(\x_i)$,
	\begin{equation*}
		\dist(0,\partial\psi_i^k(\x_i^{k+1}))=
		\left\Vert-\nabla s_i^k\left(\x_i^{k+1}\right)-\mP_{[L_i^k,U_i^k]}\left(-\nabla s_i^k(\x_i^{k+1})\right)\right\Vert_\infty\le\varepsilon_i^k.
	\end{equation*}
	We set $\varepsilon_i^k\equiv10^{-10}$ for the alternative algorithms in Table 
	{\bf\ref{table_1}}. The algorithm proposed in this paper is denoted by
	DPMM-$\varepsilon_i^k$ with some value of $\varepsilon_i^k$.
	The optimal value $F^*$ and the optimal solution $\x^*$ (if unique) of these two  
	examples are calculated by the convex optimization toolbox CVX with the precision 
	tuned to `best'.
	\subsection{Example 1 }
	Considering the following example of \eqref{P} with $m=20$, $n_i=p=3$, 
	$\forall \ i\in\V$, $q=0$,
	\begin{equation}\label{example_1}
		\begin{aligned}
			&\min_\x F(\x)=\sum_{i=1}^m\bigg(\log\left(1+\exp\left(a_i^\top\x_i\right)\right)
			+\lambda_i\Vert\x_i\Vert_1\bigg)\\
			&{\rm s.t.}\quad \sum_{i=1}^mA_i\x_i=0,
			\ l_i\le\x_i\le u_i,\ \forall i\in\V,
		\end{aligned}
	\end{equation}
	where data $a_i, A_i, l_i, u_i$ are randomly generated, $\lambda_i=i/m^2$. This linearly constrained logistic regression problem with 
	$\ell_1$-norm regularization is often encountered in machine learning.
	
	We execute the DPMM-$\varepsilon_i^k$ with $\varepsilon_i^k=10^{-10}$, Tracking-ADMM \cite{Falsone20}, PDC-ADMM \cite{Chang16}, and NECPD \cite{Su22} for example \eqref{example_1} on a standard PC. All the algorithm parameters are fine-tuned in their theoretical ranges. We let all the algorithms start from the same initial decision vector. 
	
	In Fig. {\bf\ref{fig_2}}, we plot the objective residuals and the constraint violations generated by the aforementioned four algorithms during 1000 iterations.
	The objective residual is defined as $\vert F(\x^k)-F^*\vert/\vert F^*\vert$.
	The constraint violation is $\Vert\sum_{i=1}^mA_i\x_i^k\Vert_{\infty}$.
	\begin{figure}[htbp]
		\centering
		\subfigure[Evolution of the objective residuals.]
		{
			\includegraphics[width=0.48\textwidth]{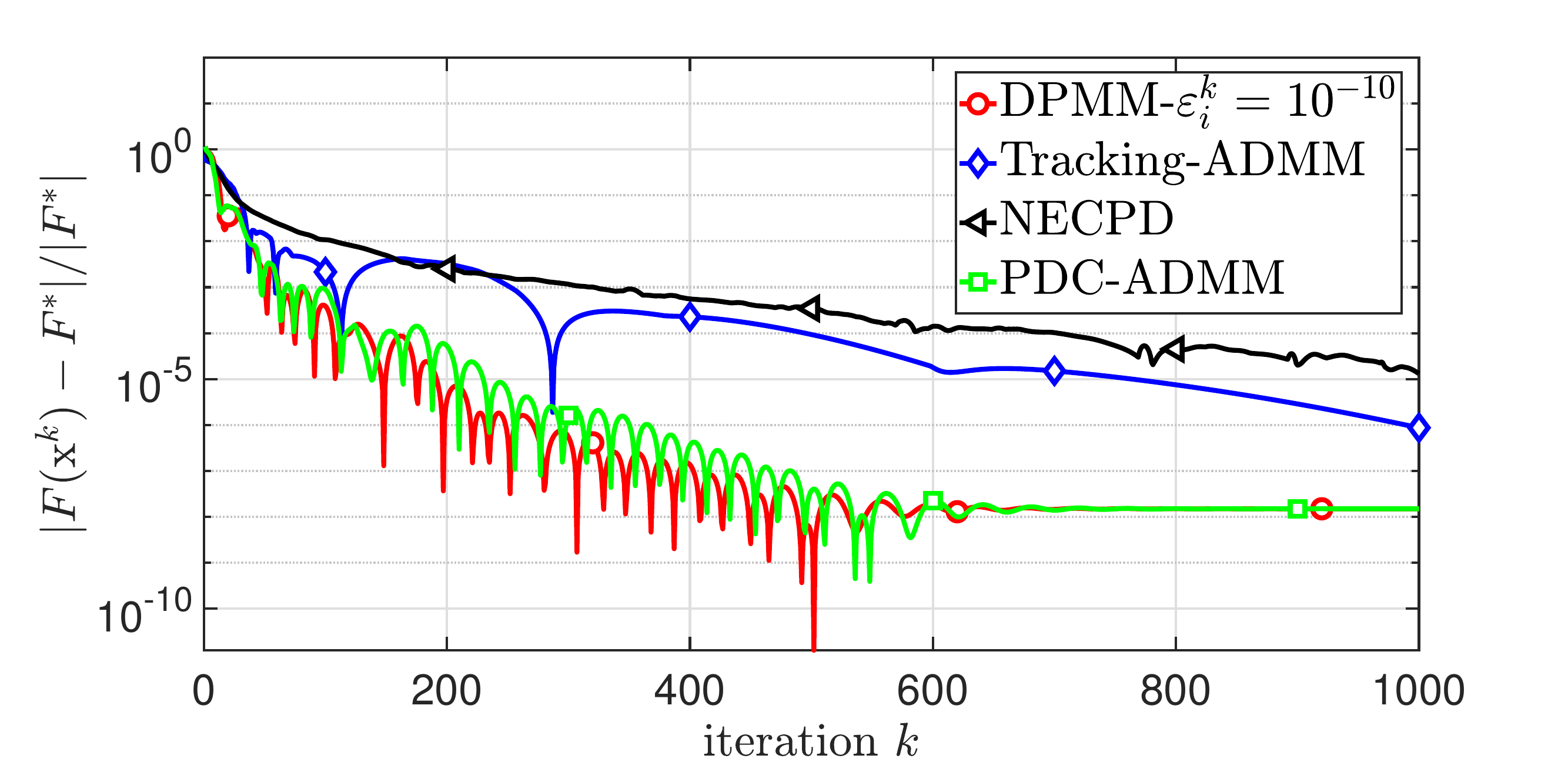}
			\label{fig_2a}
		}
		\subfigure[Evolution of constraint violations.]
		{
			\includegraphics[width=0.48\textwidth]{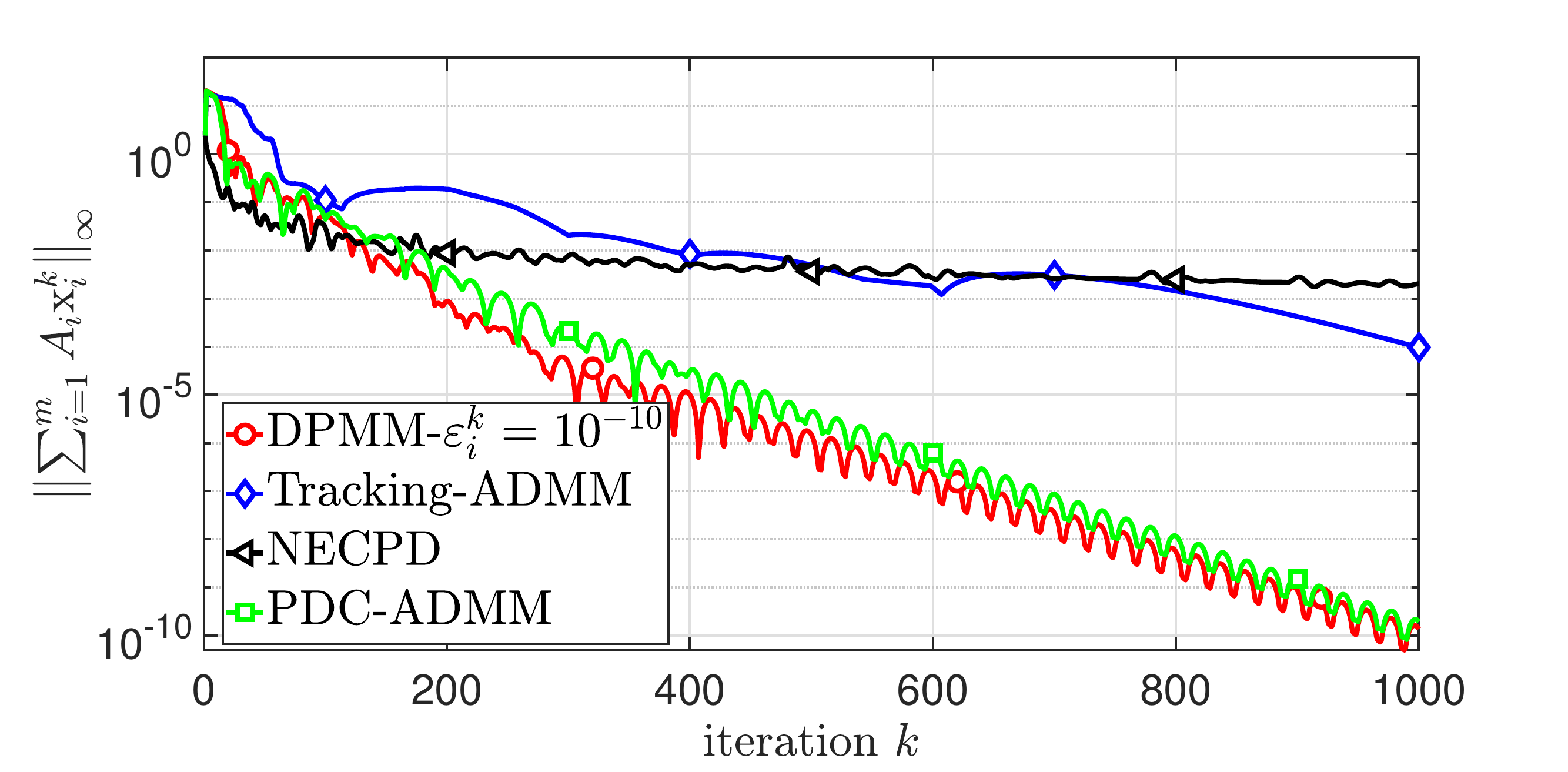}
			\label{fig_2b}
		}
		\caption{Convergence performance of algorithms}
		\label{fig_2}
	\end{figure}
	
	In Fig. {\bf\ref{fig_2}}, the Tracking-ADMM and NECPD algorithms exhibit near convergence speed for solving Example \eqref{example_1}. It is well-known that the quadratic proximal term $\frac{\gamma^k}{2}\Vert\x_i-\x_i^k\Vert^2_2$ of the objective of a subproblem brings significant benefits to the convergence of algorithms. However, the coefficient $\gamma^k$ of the NECPD algorithm tends to infinity at a rate of $O(k)$. As the algorithm runs, a large value of $\gamma^k$ may slow down its convergence speed. Tracking-ADMM incorporated the idea of gradient-tracking \cite{DIGing}, where the connected graph with few edges may delay the consensus of the variable values it tracks, thereby leading to slow convergence.
	In contrast, over a low connectivity graph ($20$ nodes and $20$ edges),
	DPMM and PDC-ADMM demonstrate a fast convergence speed with respect to optimality and feasibility.
	
	In communication costs, DPMM only requires one round of communication to
	exchange agents' estimates $\y_i^k$ of the Lagrange multiplier 
	at each iteration. Tracking-ADMM also communicates one round but exchanges two variable values. The number of communication rounds per iteration for PDC-ADMM and 
	NECPD are $2$ and $d-1$, respectively, where $d$ is the degree of the minimal 
	polynomial of the adjacency matrix $W$. Communication-wise, DPMM is more efficient than these three algorithms. 
	
	\subsection{Example 2}
	Considering the following example of \eqref{P} with $m=20$, $n_i=p=3$, $q=1$,
	\begin{equation}\label{example_2}
		\begin{aligned}
			&\min_\x F(\x)=\sum_{i=1}^m\bigg
			(\frac{1}{2}\Vert C_i\x_i-d_i\Vert_2^2+\lambda_i\Vert\x_i\Vert_1\bigg)\\
			&{\rm s.t.}\quad \sum_{i=1}^mA_i\x_i=b,
			\ l_i\le\x_i\le u_i, \forall \ i\in\V.\\
			&\hspace{2.2em}\sum_{i=1}^m\log\left(1+\exp\left(a_i^\top\x_i\right)\right)
			\le f.     
		\end{aligned}
	\end{equation}
	This constrained LASSO problem is also considered by \cite{WuXY22}.
	Each matrix $C_i$ is symmetric positive definite and randomly generated.
	$b=\sum_{i=1}^mA_i\xi_i$ with $\xi_i$ being randomly chosen from a uniform distribution over the box $[l_i,u_i]$. To ensure Assumption {\bf\ref{ass_3}} be satisfied, we take $f>\sum_{i=1}^m\log\left(1+\exp\left(a_i^\top\xi_i\right)\right)$.
	$\lambda_i=i/m^2$, $A_i,\ a_i,$ and $d_i$ are randomly generated, $\forall\ i=1,2,\dots,m$.
	
	We run the IPLUX algorithm \cite{WuXY22} and DPMM-$\varepsilon_i^k$ with $\varepsilon_i^k=10^{-10}, 1/k^2$, and $1/k^{1.2}$ on a standard PC to solve example \eqref{example_2}. Again, we choose the same initial decision vector $\x^0$ for both algorithms and then compare the practical convergence performance with their parameters tuned within respective theoretical ranges.
	
	\begin{figure}[htbp]
		\centering
		\subfigure[Evolution of objective residual.]
		{
			\includegraphics[width=0.48\textwidth]{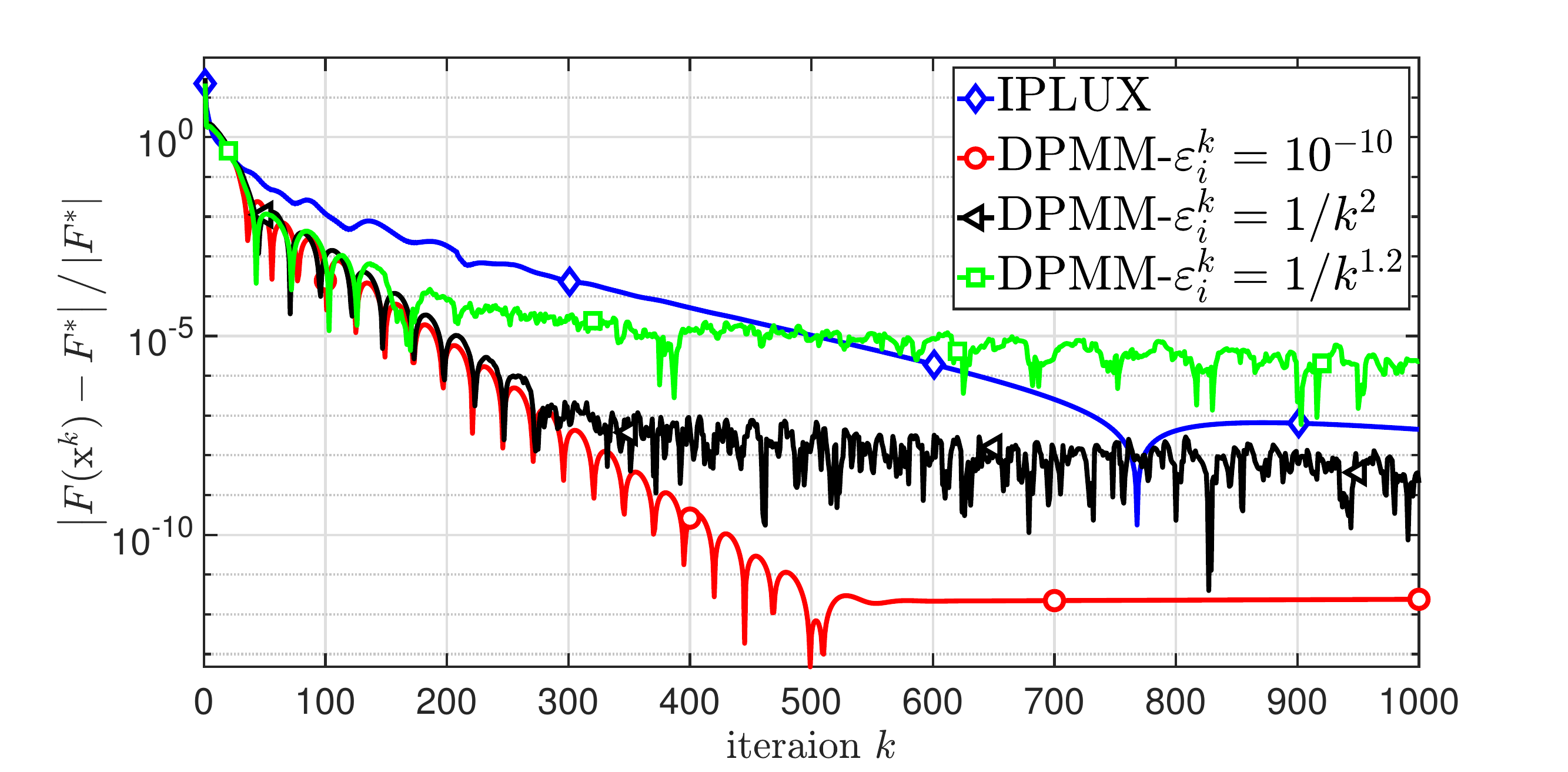}
			\label{fig_3a}
		}
		\subfigure[Evolution of constraint violation.]
		{
			\includegraphics[width=0.48\textwidth]{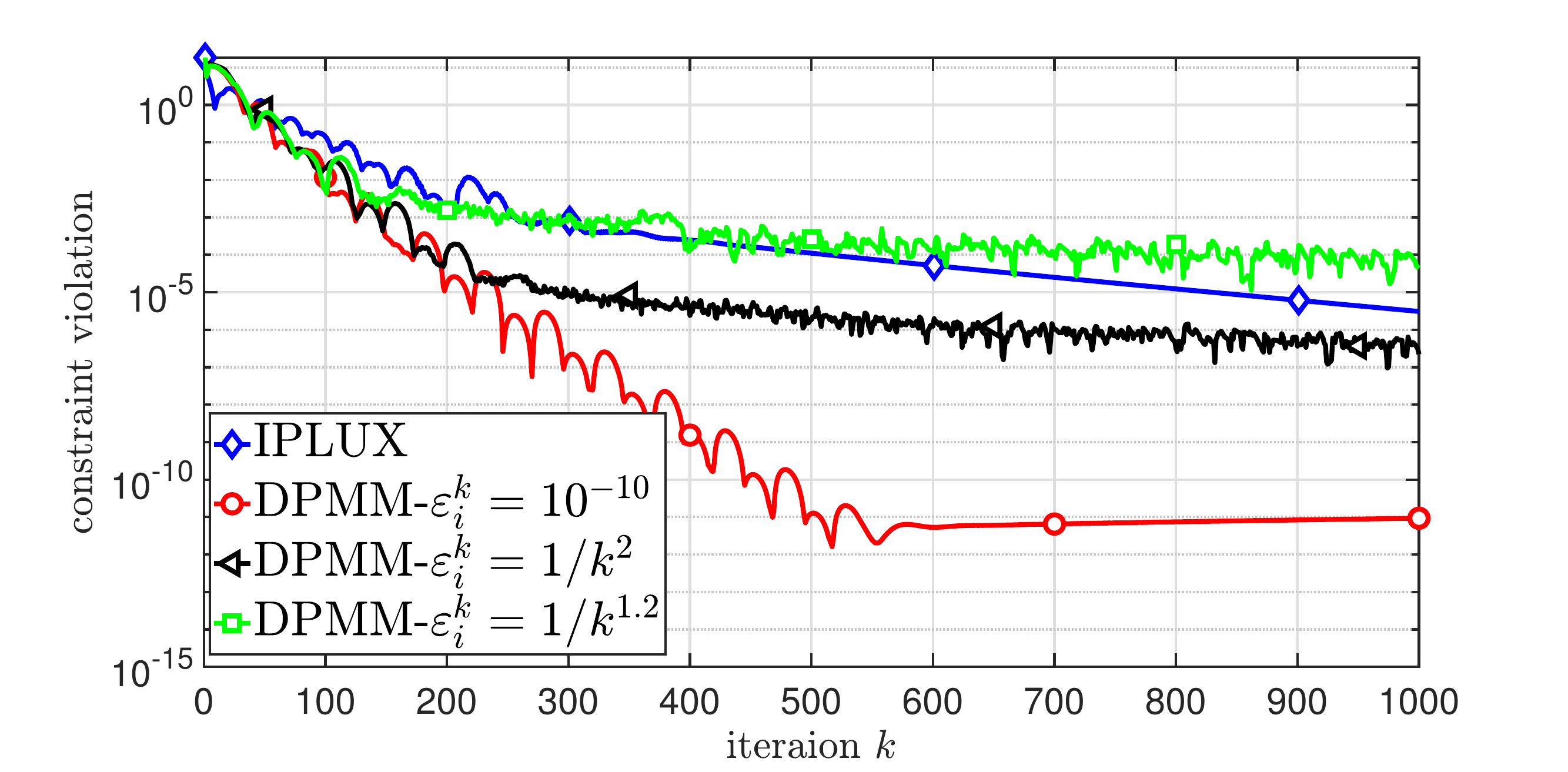}
			\label{fig_3b}
		}
		\subfigure[Evolution of optimality error.]
		{
			\includegraphics[width=0.7\textwidth]{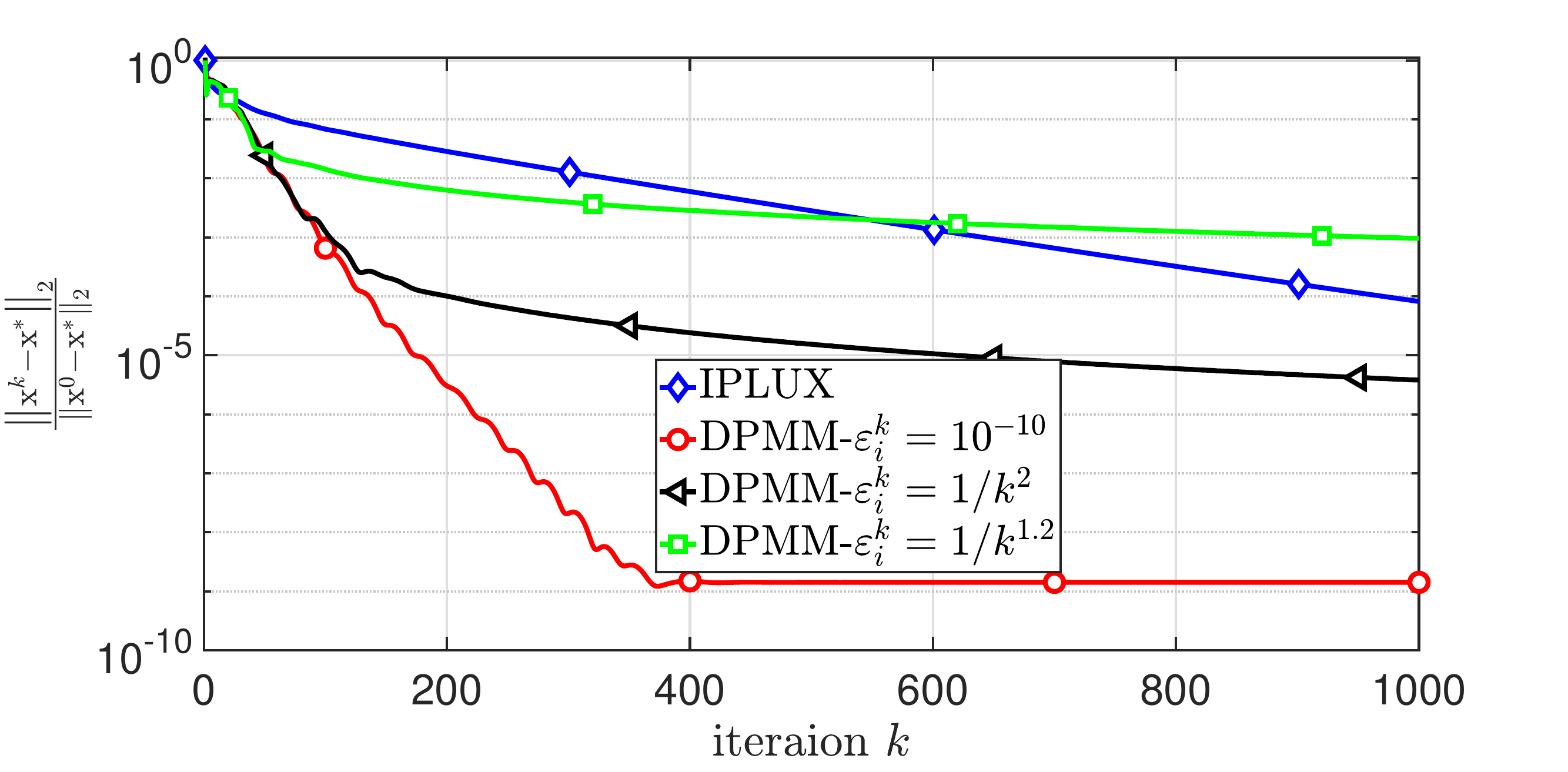}
			\label{fig_3c}
		}
		\caption{Convergence performance of algorithms}
		\label{fig_3}
	\end{figure}
	
	In Fig. {\bf\ref{fig_3}}, we plot the objective residual, constraint violation,
	and optimality error of running DPMM-$\varepsilon_i^k$ and IPLUX for 1000 iterations. The constraint violation is the sum of $\Vert\sum_{i=1}^mA_i\x_i^k-b\Vert_{\infty}$ and
	$\max\big\{\sum_{i=1}^m\log\left(1+\exp\left(a_i^\top\x_i\right)
	\right)-f,0\big\}$. The optimality error is 
	$\Vert\x^k-\x^*\Vert_2/\Vert\x^0-\x^*\Vert_2$.
	
	As illustrated in Fig. {\bf\ref{fig_3a}} and Fig. {\bf\ref{fig_3b}}, the IPLUX algorithm effectively solves Example 2, reaching a theoretical convergence rate of $O(1/k)$ concerning the objective residuals and constraint violations, while Fig. 
	{\bf\ref{fig_3c}} shows the slow progression of its decision variable sequence $\{\x_k\}$ towards the optimal solution $\x^*$. It can be summarized from Fig. {\bf\ref{fig_3}} that the DPMM-$\varepsilon_i^k$ algorithm has a significant convergence speed advantage over IPLUX for solving the constrained LASSO problem \eqref{example_2}, in terms of the optimal errors, objective residuals, and constraint violations. 
	
	Fig. {\bf\ref{fig_3}} also indicates that the DPMM-$\varepsilon_i^k$ can be applied to a wider range of practical optimization problems, as we can see, with a dynamic precision $\varepsilon_i^k=1/k^2$ to compute the subproblem, running the DPMM-$\varepsilon_i^k$ for about 500 iterations yields a trial solution reaching an accuracy of $10^{-5}$, in terms of the objective residual, constraint violation, and optimality error. Of course, the faster the computational error $\varepsilon_i^k$ of the subproblem decays to 0, the faster the DPMM-$\varepsilon_i^k$ algorithm converges, which is also reflected in Fig. {\bf\ref{fig_3}}.
	
	The algorithmic parameter $\tau=\alpha+\gamma\lambda^2$ of IPLUX is influenced by the objective and constraint functions of the solved optimization problem, i.e., 
	$\alpha\ge L_f+L^2$ and $\lambda\ge\Vert A^s\Vert_2$, where $L_f$ is the gradient Lipschitz constant of the smooth part of the objective function, $L_h$ is the Lipschitz constant of the coupled inequality constraint function $\sum_{i=1}^mh_i(\x_i)$, and $A^s$ is the matrix in the sparse coupled linear constraint considered by IPLUX (which is merged with the global coupled constraint $\sum_{i=1}^mA_i\x_i=b$ in this paper). This means that IPLUX requires some prior knowledge of the solved optimization problems, e.g., knowledge of these Lipschitz constants mentioned above, which may need all agents' communication to achieve.  However, this concern does not exist for DPMM, reviewing Proposition {\bf\ref{PROP_1}} and Remark {\bf\ref{rem_2}}, the parameter selections of DPMM can be done without any knowledge of the Lipschitz constants associated with the solved optimization problem and the network structure of the connected undirected graph, which implies that DPMM is robust. 
	
	Communication-wise, recall that agents exchange the Lagrange multiplier estimates
	$\widehat{\y}_i^k$ once per iteration in the DPMM algorithm, while IPLUX 
	doubles such communication costs, which indicates that DPMM is more efficient.
	
	\section{Conclusion}\label{sec_6}
	In this paper, we have developed a distributed proximal method of multipliers, 
	referred to as DPMM, to address coupled constrained convex optimization problems over 
	a fixed undirected connected network. We demonstrate its primal-dual convergence and 
	an $o(1/k)$ rate for the first-order optimality residual under general convexity. Furthermore, under the structural assumption, DPMM converges linearly. Numerical simulations reveal that the proposed DPMM algorithm is efficient for \eqref{P}, comparing with some alternative distributed constrained optimisation algorithms.
	
	\bibliographystyle{unsrt}
	\bibliography{ref_DPMM}
\end{document}